\documentclass[a4paper,10pt]{article}

\usepackage[T1]{fontenc}
\usepackage{amsmath}
\usepackage{amssymb}
\usepackage{graphicx}
\usepackage{color}
\usepackage{amssymb}
\usepackage{latexsym}
\usepackage{dsfont,pifont}
\usepackage{bbold}

 \newcommand{\vc}[1]{{\boldsymbol #1}}
 \newcommand{\vpi}{\vc\pi}
\newcommand{\vone}{\vc 1}
\newcommand{\vxi}{\vc \xi}

\newcommand{\real}{\mathrm{Re}}
\newcommand{\ve}{\varepsilon}
\newcommand{\bs}{\boldsymbol}
\newcommand{\op}{\oplus}
\newcommand{\om}{\ominus}
\DeclareMathOperator*{\spec}{sp}
\DeclareMathOperator*{\diag}{diag}
\newtheorem{theorem}{Theorem}[section]
\newtheorem{assumption}[theorem]{Assumption}
\newtheorem{rem}[theorem]{Remark}

\newcommand{\qed}{\hfill $\square$}
\newenvironment{proof}{
      \noindent {\bf Proof }}{\qed
      \vspace{0.25\baselineskip}
}

\newcommand{\vligne}[1]{\begin{bmatrix} #1 \end{bmatrix}}
\newcommand{\dd}{\mathrm{d}}

\begin{document}

\title{Perturbation analysis of Markov modulated fluid models}
\author{
Sarah Dendievel\thanks{Ghent University, Department of
  Telecommunications and Information Processing, SMACS Research Group,
  Sint-Pietersnieuwstraat 41, B-9000 Gent, Belgium, \texttt{Sarah.Dendievel@UGent.be}}
\and 
Guy Latouche\thanks{Universit\'e libre de Bruxelles,  Facult\'e des sciences, CP212, Boulevard du Triomphe 2, 1050 Bruxelles,
  Belgium,        \texttt{latouche@ulb.ac.be}}
       }


\maketitle
\begin{abstract}
We consider perturbations of positive recurrent Markov modulated fluid models.
  In addition to the infinitesimal generator of the
  phases, we also perturb the rate matrix, and analyze the effect
  of those perturbations on   the matrix of first return   probabilities to the initial level.
   Our main contribution is the construction of a
  substitute for the matrix of first return probabilities, which
  enables us to analyze the effect of the perturbation under
  consideration. 

\noindent \underline{Keywords}: Markov modulated fluid models; Perturbation analysis; First return probabilities.
\end{abstract}

\section{Introduction}

Most mathematical models have input parameters that are typically estimated from the real world data.
Since the parameters in the modeled system represent quantities that can suffer from small errors,
it is natural to analyze how the performance measures are affected by small changes in the parameters.
Using the structural properties of the model, it becomes possible to assess the impact of perturbations on the key matrices of the underlying process by providing computationally feasible solutions along with probabilistic
interpretation.

Markov modulated fluid models appeared in the 1960s to study the
continuous-time behavior of queues and dams, an early paper being
Loynes \cite{loynes1962continuous}.  In the eighties, Markovian fluid
models started to be more extensively investigated, in particular
their stationary density, see for instance Rogers
\cite{rogers1994fluid} and Asmussen \cite{asmussen1995stationary}.
The importance of the matrix of first return probabilities has been
demonstrated in Ramaswami~\cite{ram99} and its computation has
attracted much attention, see Bean {\it et al.}
\cite{bean2005algorithms} and Bini {\it et al.}
\cite{bini2006solution}.  One may derive from $\Psi$, {\it the matrix of first return probabilities form above}, important
performance measures of the model, such as the stationary density of
the level of the fluid model.

The model $\{(X(t),\varphi(t)):t\in\mathbb{R}^{+}\}$ is described as
follows: $\varphi(t)$ is a Markov chain, with finite state space
$\mathcal{S}$, it is called the {\it{phase}} process; $ X(t)$ is a
continuous function, called the {\it{level}}.  The evolution of the
level is continuous and may be expressed as
\begin{align}
\nonumber
X(t)&= Y(t) + 
\sup_{0\leq s \leq t}\left\{\max\left(0,-Y(s)\right)\right\}
\\
\text{where} \quad 
Y(t)&=Y(0)+\int_{0}^{t}c_{\varphi(s)}\dd s,
\label{XofTfor_MMFM}
\end{align}
so that it varies linearly with rate
$c_{i}$ when $\varphi(t)=i$, $i\in \mathcal{S}$.
We partition $\mathcal{S}$ into $\mathcal{S}_{+}\cup \mathcal{S}_0 \cup \mathcal{S}_{-}$
with $\mathcal{S}_{+}=\{i\in \mathcal{S}:c_{i}>0\}$,
$\mathcal{S}_{0}=\{i\in \mathcal{S}:c_{i}=0\}$  and $\mathcal{S}_{-}=\{i\in \mathcal{S}:c_{i}<0\}$.
The infinitesimal generator of the phase process
is denoted by $A$ and is written,
possibly after permutation of rows and
columns,  as 
\begin{equation}
A=\left[\begin{array}{ccc}
A_{++} & A_{+0}& A_{+-}
\\
A_{0+} & A_{00}& A_{0-}
\\
A_{-+}  & A_{-0} & A_{--}
\end{array}\right],
\label{GeneratorA}
\end{equation}
and the {\it{rate matrix}} is denoted by
\begin{equation}\label{FluidRateMatrix}
C=
\left[\begin{array}{ccc}
C_{+} & &\\
&C_{0}&
\\
& & C_{-}
\end{array}\right]
\end{equation}
{with}
$C_+=  \mbox{diag}(c_i:i\in \mathcal{S}_{+} )$,
$C_-=  \mbox{diag}(c_i:i\in \mathcal{S}_{-} )$
and 
$C_0$
is a null matrix.
Throughout the paper, we make the following assumption.
\begin{assumption}
The Markov modulated fluid model is positive recurrent, that is,
$\boldsymbol{\vxi}C\vone <0$, where
$\vc \xi$ is the stationary probability vector defined for $i,j\in\mathcal{S}$ by
\begin{equation}\label{defdeXistatphase}
\xi_i=\lim_{t\rightarrow \infty}\mathbb{P}\left[\varphi(t)=i | \varphi\left(0\right)=j\right],
\end{equation}
and
is the unique solution of the equation
$\boldsymbol{\vxi}A=0$
such that
$\boldsymbol{\vxi1}=\boldsymbol{1}$,
where $\boldsymbol{1}$ denotes the column vector of 1's.
\end{assumption}

A key matrix for Markov modulated fluid models is the matrix $\Psi$ of
{\it first return probabilities to the initial level from above}, with dimensions
$|\mathcal{S}_{+}|\times |\mathcal{S}_{-}|$, and components
\begin{equation}\label{PsiBasis}
\Psi_{ij}=\mathbb{P}\left[\tau_{-}<\infty,\varphi\left(\tau_{-}\right)=j|Y\left(0\right)=0,\varphi\left(0\right)=i\right],
\end{equation}
where $\tau_{-}=\inf\{t>0:Y(t)<0\}$,
$i\in \mathcal{S}_{+}$
and
$j\in \mathcal{S}_{-}$.
By Rogers \cite[Theorem 1]{rogers1994fluid},
 $\Psi$ is the minimal nonnegative solution of the Riccati equation
\begin{equation}\label{RiccEquaNullph}
C_{+}^{-1}Q_{+-}+C_{+}^{-1}Q_{++}X+X\left|C_{-}^{-1}\right|Q_{--}+X\left|C_{-}^{-1}\right|Q_{-+}X=0,
\end{equation}
where
$|C_{-}^{-1}|$ denotes the entrywise absolute value of $C_{-}^{-1}$
and
\begin{align}\label{EqRefQ}
\left[\begin{array}{cc}
Q_{++} & Q_{+-}\\
Q_{-+} & Q_{--}
\end{array}\right]
=&\left[\begin{array}{cc}
A_{++} & A_{+-}\\
A_{-+} & A_{--}
\end{array}\right]
+\left[\begin{array}{c}
A_{+0}\\
A_{-0}
\end{array}\right]\left(-A_{00}^{-1}\right)\left[\begin{array}{cc}
A_{0+} & A_{0-}\end{array}\right].
\end{align}
Similarly, the {\it{matrix $\hat{\Psi}$ of first return probabilities to the initial level from below}} has components
\begin{equation*}
\hat{\Psi}_{ij}=\mathbb{P}\left[\tau_{+}<\infty,\varphi\left(\tau_{+}\right)=j|Y\left(0\right)=0,\varphi\left(0\right)=i\right],
\end{equation*}
where $\tau_{+}=\inf\{t>0:Y(t)>0\}$,
$i\in \mathcal{S}_{-}$
and
$j\in \mathcal{S}_{+}$, it satisfies a Riccati equation similar to 
\eqref{RiccEquaNullph}.
The present article focuses on the perturbation  analysis of $\Psi$
only, as the analysis for $\hat{\Psi}$ is similar.

Two other important matrices are 
\begin{align}
U&=|C_{-}^{-1}|Q_{--}+ |C_{-}^{-1}| Q_{-+} \Psi, \label{matrixU}
\\
K&=C_{+}^{-1}Q_{++}+\Psi |C_-^{-1}| Q_{-+}.   \label{matrixK}
\end{align}
The matrix $U$ is the infinitesimal generator of the process of downward record and is such that for $i,j\in\mathcal{S}_-$,
 $(e^{Ux})_{ij}$
is the probability that, starting from $(y,i)$, for any $y$,
the process reaches level $y-x$ in finite time and that $(y-x,j)$ is the first state visited in level $y-x$.
The matrix $K$ defined in \eqref{matrixK} is also an important matrix for Markov modulated fluid models and appears in the sationary density of the fluid model, see Section \ref{application}.

For a long time there has been a recurrent interest in perturbation analysis, see for instance
Cao and Chen \cite{cao1997perturbation},
Heidergott, {\it et al.} \cite{heidergott2010series},
Antunes {\it et al.}
\cite{antunes2006perturbation}. 
In this paper, we analyze the perturbation of Markov modulated fluid models.
When  the infinitesimal generator \eqref{GeneratorA}  of the phases is
perturbed into $A(\varepsilon) = A + \varepsilon \tilde{A} $, the
analysis follows the usual path:  the perturbed first return
probability matrix $\Psi(\ve)$ is  shown to be analytic, and
computable equations are readily obtained for the derivatives of $\Psi(\ve)$.  
We focus on the first order derivative
\[
\Psi^{(1)} = \frac{\dd \Psi (\varepsilon)}{\dd
  \varepsilon}\bigg|_{\varepsilon=0}
\]
of a perturbed Markov modulated fluid model as it provides a good
approximation of the effect of the perturbation on the system when
compared to the unperturbed system.  Furtermore, we are interested in
the structures and going beyond the first derivative is rather
computational and does not bring much more information.
 
We also analyze the effect on $\Psi$ of perturbations of the rate
matrix \eqref{FluidRateMatrix}.  When $C$ is perturbed as
$C(\varepsilon)=C+\varepsilon \tilde{C}$, phases of $\mathcal{S}_{0}$
may be transformed into phases of $\mathcal{S}_{+}$ or
$\mathcal{S}_{-}$ in the perturbed model, with the consequence that a
perturbation of the rates $c_i$ appearing in \eqref{XofTfor_MMFM} may
modify the structure of $\Psi(\varepsilon)$ as the dimensions are not
the same as those of $\Psi$.  Clearly, the comparison between the
matrices $\Psi(\varepsilon)$ and $\Psi$ requires more care.


We do not consider cases where both the generator $A$ and the rate matrix $C$ are perturbed, as our results show that this may be done, at the cost of increased complexity in the expressions obtained.


In Section \ref{AperturbedSection}, we analyze perturbations of the
infinitesimal generator of the phases.
In Section \ref{section3}, we analyze perturbations on the rate matrix
$C$ in four different cases.
In Section \ref{Case1} we assume that the phases of $\mathcal{S}_0$ are unaffected by the perturbation.
In Sections~\ref{AffectedCase+}--\ref{GenCaseMig} we examine what
happens when the phases of $\mathcal{S}_0$ are affected by the
perturbation. We propose an adapted version of $\Psi$ which enables
the analysis of the effect of the perturbation under consideration.
We decompose the analysis in three subsections for the sake of
clarity: firstly, we assume that all the phases in $\mathcal{S}_0$
become phases of $\mathcal{S}_+$ after perturbation, next, we assume
that they all become phases of $\mathcal{S}_-$ after perturbation,
finally, we assume that the phases in $\mathcal{S}_0$ are split
between $\mathcal{S}_+$ and $\mathcal{S}_-$.  The general approach is
the same in the three cases but the details differ and become much
more involved in the last.
As an application, we
derive in Section \ref{application} the first order approximation of
the stationary density of a perturbed fluid model.
In Section \ref{Illustr}, we provide a numerical illustration.


\section{Perturbation of the infinitesimal generator}\label{AperturbedSection}

In this section, the infinitesimal generator $A$ is perturbed and
becomes 
\begin{equation}\label{TransMatrPertFluid}
A(\varepsilon) = A + \varepsilon \tilde{A} ,
\end{equation}
where
\begin{equation}\label{defdepertubTtildFluid}
\tilde{A}  =
\begin{bmatrix}
\tilde{A}_{++}  & \tilde{A}_{+0} &  \tilde{A}_{+-}  \\ 
\tilde{A}_{0+}  & \tilde{A}_{00} &  \tilde{A}_{0-}  \\ 
\tilde{A}_{-+}  & \tilde{A}_{-0}  & \tilde{A}_{--} 
\end{bmatrix}, 
\end{equation}
$\widetilde{A} \vc 1 = 0$, and we assume that $A(\ve)$ is an
irreducible infinitesimal generator for $\varepsilon$ sufficiently
small in a neighborhood of~0.

The matrix $\Psi(\varepsilon)$ of first return probabilities for the perturbed model is the minimal nonnegative solution of the Riccati equation
\begin{align}
&C_{+}^{-1}Q_{+-}(\varepsilon)+C_{+}^{-1}Q_{++}(\varepsilon)X     
+X\left|C_{-}^{-1}\right|Q_{--}(\varepsilon)+X\left|C_{-}^{-1}\right|Q_{-+}(\varepsilon)X=0,
\label{RiccEquaNullphPert}
\end{align}
where $Q(\varepsilon)$ is defined by (\ref{EqRefQ}), with $A(\ve)$
replacing $A$.  We write 
\begin{align*}
& \left[\begin{array}{cc}
Q_{++}(\ve) & Q_{+-}(\ve)\\
Q_{-+}(\ve) & Q_{--}(\ve)
\end{array}\right]
 =\left[\begin{array}{cc}
Q_{++}+\ve\tilde{Q}_{++} & Q_{+-}+\ve\tilde{Q}_{+-}\\
Q_{-+}+\ve\tilde{Q}_{-+} & Q_{--}+\ve\tilde{Q}_{--}
\end{array}\right] +O(\ve^2).
\end{align*}

\begin{theorem}\label{ThmforAPerturbed}
The matrix $\Psi(\varepsilon)$ of first return probabilities, minimal nonnegative solution to \eqref{RiccEquaNullphPert},   for the perturbed model is analytic in a neighbourhood of zero. Furthermore, $\Psi^{(1)} $ is the unique solution of the Sylvester equation
\begin{align}
K X +X U 
=&
-C_+^{-1}\tilde{Q}_{+-}
-C_{+}^{-1}\tilde{Q}_{++}\Psi
-\Psi |C_{-}^{-1}|\tilde{Q}_{--}
- \Psi  |C_{-}^{-1}| \tilde{Q}_{-+} \Psi,
\label{SylvesterEqPsi1Genpert}
\end{align}
where 
$K$ and $U$
are defined in (\ref{matrixK}) and (\ref{matrixU}).
\end{theorem}

\begin{proof}
Define the continuous operator 
\begin{align*}
F(\varepsilon,\mathcal{X}) = \  &C_{+}^{-1}Q_{+-}(\varepsilon)
+C_{+}^{-1}Q_{++}(\varepsilon)\mathcal{X}
 +\mathcal{X}\left\vert C_{-}^{-1}\right\vert Q_{--}(\varepsilon)
+\mathcal{X}\left\vert C_{-}^{-1}\right\vert Q_{-+}(\varepsilon)\mathcal{X}.
\end{align*}
We have  $F(0,\Psi)=0$ and
${\partial}_{\mathcal{X}}F(\varepsilon,\mathcal{X})$ exists in a neighborhood of  $(0,\Psi)$ and is continuous  at $(0,\Psi)$.
For $Y,H\in \mathbb{R}^{|\mathcal{S}_+| \times |\mathcal{S}_-|}$, the equation
\begin{equation*}
\left.{\partial}_{\mathcal{X}}F(\varepsilon,\mathcal{X})\right|_{\varepsilon =0,\mathcal{X}=\Psi}(Y)
=H,
\end{equation*}
is equivalent to 
 the Sylvester equation
\begin{equation}\label{SylvEqPsiDiff}
KY
+ Y U
=H.
\end{equation}
From Rogers \cite{rogers1994fluid} and Govorun {\it{et al.}}
\cite{govorun2013stability}, we have
$\spec(K)\in\{z\in\mathbb{C}:\real({z}) < 0\}$ and
$\spec(-U)\in\{z\in\mathbb{C}:\real({z})\geq 0\}$.  Thus, $K$ and $-U$
have no common eigenvalue and, by Lancaser and Tismenetsky \cite[page
414]{lancaster1985m}, \eqref{SylvEqPsiDiff} has a unique solution, so
that 
$\left.{\partial}_{\mathcal{X}}F(\varepsilon,\mathcal{X})\right|_{\varepsilon
  =0,\mathcal{X}=\Psi(0)}$
is a nonsingular operator.  We conclude that $\Psi(\varepsilon)$ is
analytic at zero by the Implicit Function Theorem.  
%
\end{proof}

\begin{rem}\rm{
It immediately results from  Xue {\it{et al.}} \cite[Theorem
2.2]{xue2012accurate} that small {\em relative} changes to the entries
of $Q$ induce small {\em relative} differences between $\Psi$ and
$\Psi(\varepsilon)$.  The bounding coefficient matrix in \cite[Eqn. 
(2.12)]{xue2012accurate} is the solution of a Sylvester equation with the same
coefficients $K$ and $U$ as in (\ref{SylvesterEqPsi1Genpert}) and a
different right-hand side.}
\end{rem}

\section{Perturbation of the rate matrix}\label{section3}

Define
  \begin{equation}\label{Cperturbation}
    C(\varepsilon)=C+\varepsilon \tilde{C}
  \end{equation}
with $\tilde C = \diag (\tilde c_i : i \in \mathcal S)$, partitioned as
\begin{equation}
\tilde{C}=
\left[\begin{array}{ccc}
\tilde{C}_{+} & &\\
&\tilde{C}_{0}&
\\
& & \tilde{C}_{-}
\end{array}\right]
\end{equation} 
where the orders of $\tilde{C}_+$, $\tilde{C}_0$ and $\tilde{C}_-$ are
equal to those of $C_+$, $C_0$ and $C_-$, respectively.  Assume that
$\ve$ is small enough so that the diagonal elements of 
$C_{+}(\varepsilon)$ are strictly positive and those of  
$C_{-}(\varepsilon)$ strictly negative.

We analyze separately the cases $\tilde{C_0}=0$ (in Section
\ref{Case1}) and $\tilde{C_0}\neq0$.  If $\tilde{C_0}\neq0$, the
perturbation has the effect of changing null phases into non-null
phases.  To simplify the presentation, we suppose at first that all
phases of $\mathcal{S}_0$ become phases of the same non-null subset
$\mathcal{S}_+$ after perturbation.  This is analyzed in Section
\ref{AffectedCase+}.  In Section \ref{Casemigrto-}, we treat the case
where all the phases of $\mathcal{S}_0$ become phases of
$\mathcal{S}_-$ after perturbation.  Finally, we assume in Section
\ref{GenCaseMig}  that the
phases in $\mathcal{S}_0$ are split partially into $\mathcal{S}_+$ and
into $\mathcal{S}_-$.
 
Clearly, Section~\ref{GenCaseMig} covers the cases analyzed in
Sections~\ref{AffectedCase+} and \ref{Casemigrto-}.  It is useful,
nevertheless, to proceed through the special cases first, for which
the results are easier to follow. In various remarks, we emphasize the
unity of treatment.

The Implicit Function Theorem applies in all cases to prove the
 analyticity of $\Psi(\ve)$, although details become more involved as
 we proceed from the simplest  to the most general case. 
 We show this in Theorem \ref{Theorem32} and Theorem \ref{ThmFour}  and we omit the details for 
Theorem \ref{ThmFive}.

\subsection{Phases in $\mathcal{S}_0$ unaffected}\label{Case1}

Assume that  
 $\tilde{C}_{0}=0$  so that $C_0(\varepsilon)=0$ as well.
The matrix $\Psi(\varepsilon)$ of first return probabilities for the perturbed model is the minimal nonnegative solution of the Riccati equation
\begin{align}
&C_{+}^{-1}(\varepsilon)Q_{+-}+C_{+}^{-1}(\varepsilon)Q_{++}X
+X\left|C_{-}^{-1}(\varepsilon)\right|Q_{--}+X\left|C_{-}^{-1}(\varepsilon)\right|Q_{-+}X=0.\label{Ricceqfirstcase}
\end{align}
The next Theorem is proved by applying to  \eqref{Ricceqfirstcase} the same argument as in Theorem  \ref{ThmforAPerturbed}.

\begin{theorem}
Assume $C(\varepsilon)=C+\varepsilon \tilde{C}$, with $\tilde{C}_{0}=0$. The matrix $\Psi(\varepsilon)$ of first return probabilities for the perturbed model is analytic at zero and may be written as
\begin{equation*}
{\Psi}(\varepsilon)={\Psi}+\ve\Psi^{(1)}+O(\ve^2),
\end{equation*} 
where
$\Psi$ is the minimal non-negative solution to \eqref{RiccEquaNullph} and
 $\Psi^{(1)}$  is the unique solution of the Sylvester equation
\begin{equation}
K X +X U 
= 
-\Psi |C_{-}^{-1}| \tilde{C}_{-} U
-  C_{+}^{-1} \tilde{C}_{+} \Psi U,
\label{Psi1eqforNullNP}
\end{equation}
where 
$K$ and $U$ are defined in \eqref{matrixK} and \eqref{matrixU} respectively.
\hfill $\square$
\end{theorem}

\subsection{Migration of $\mathcal{S}_0$ to $\mathcal{S}_+$}\label{AffectedCase+}
Assume that $\tilde c_i >0$ for all $i$ in $\mathcal{S}_0$,  this
means that all phases of $\mathcal{S}_0$ become phases of fluid
increase after perturbation. 
 To make this explicit in our equations,
we replace the subscript  0 by the subscript $\oplus$ and write $\mathcal S_\oplus$ instead of $\mathcal 
  S_0$, etc. 
The infinitesimal generator of the phase process is written as
\begin{equation}\label{MatrixPsibeforerPerturbPlus}
A= \left[\begin{array}{c|c|c}
A_{++} & A_{+\oplus} & A_{+-}\\
\hline A_{\oplus+} & A_{\oplus\oplus} & A_{\oplus-}\\
\hline A_{-+} & A_{-\oplus} & A_{--}
\end{array}\right].
\end{equation}
After perturbation, it is partitioned as
\begin{equation}\label{MatrixPsiAfterPerturbPlus}
A= \left[\begin{array}{cc|c}
A_{++} & A_{+\oplus} & A_{+-}\\
A_{\oplus+} & A_{\oplus\oplus} & A_{\oplus-}\\
\hline A_{-+} & A_{-\oplus} & A_{--}
\end{array}\right]
\end{equation}
and the set of phases with positive rates is
$\mathcal{S}_{+} \cup \mathcal{S}_{\oplus}$.  The dimensions of the
first return probability matrix become
$(|\mathcal{S}_{+}| + |\mathcal{S}_{\oplus}|) \times |\mathcal{S}_-|$
after perturbation and $\Psi$ may not be directly compared to
$\Psi(\ve)$, the matrix of first return probabilities of the perturbed
model, which is partitioned as
\begin{equation}\label{defPsiCasePlus}
\bs{\Psi}(\varepsilon)=\left[\begin{array}{c}
{\Psi}_{+-}(\varepsilon)\\
{\Psi}_{\oplus-}(\varepsilon)
\end{array}\right].
\end{equation}
The matrix $\vc\Psi(\ve)$ is the  minimal nonnegative solution of
the Riccati equation 
\begin{align}
& \left[\begin{array}{cc}
C_{+} +\varepsilon\tilde{C}_{+}\\
 & \varepsilon\tilde{C}_{\oplus}
\end{array}\right]^{-1}
\left(
\left[\begin{array}{c}
A_{+-}\\
A_{\oplus-}
\end{array}\right]
+
\left[\begin{array}{cc}
A_{++} & A_{+\oplus}\\
A_{\oplus+} & A_{\oplus\oplus}
\end{array}\right]X
\right)\nonumber\\
&+
X\left|C_{-}+\varepsilon\tilde{C}_{-}\right|^{-1}
\left(A_{--}
+
\left[\begin{array}{cc}
A_{-+} & A_{-\oplus}\end{array}\right]
X
\right)
=0.
\label{RiccEqPsiNullPlus}
\end{align}
As we show in the next theorem, comparisons are nevertheless possible,
as $\Psi$ is immediately recognised in
the limit $\overline{\Psi}=\lim_{\ve \rightarrow 0}\bs{\Psi}(\varepsilon)$.

\begin{theorem}\label{Theorem32}
  The matrix \eqref {defPsiCasePlus} of first return probabilities for
  the perturbed model, minimal nonnegative solution of
  \eqref{RiccEqPsiNullPlus}, is analytic near zero and may be written as
\[
\bs{\Psi}(\varepsilon)=\overline{\Psi}+\ve\Psi^{(1)}+O(\ve^2),
\]
where
\begin{align}\label{Psioverline}
\overline{\Psi}
&=
\left[\begin{array}{c}
{\Psi}\\
\Psi_{\op-}
\end{array}\right]
\text{\quad and \quad}
{\Psi}^{(1)}
=
\left[\begin{array}{c}
{\Psi}_{+-}^{(1)}\\
{\Psi}_{\oplus -}^{(1)}
\end{array}\right],
\end{align}
where
$\Psi$ is given in  \eqref{RiccEquaNullph},
$\Psi_{\op-}=(-A_{\oplus\oplus}^{-1})
(A_{\oplus-}
+
A_{\oplus+}\Psi)$,
${\Psi}_{+-}^{(1)}$ is  the unique solution of the Sylvester equation
\begin{equation}\label{eqforPsiprimefirstPert}
K X+X U 
= 
-\Psi |C_{-}^{-1}| \tilde{C}_{-} U
-  C_{+}^{-1} \tilde{C}_{+} \Psi U
- P_{\oplus} U,
\end{equation}
and
\begin{align}\label{eqforPsoplusdsunrondm}
{\Psi}_{\oplus-}^{(1)}
&=
(-A_{\oplus \oplus}^{-1})\tilde{C}_{\oplus}\Psi_{\oplus-}{U}
+(- A_{\oplus \oplus}^{-1}) A_{\oplus + } \Psi_{+-}^{(1)}.
\end{align}
The matrices
$K$ and $U$ are defined in \eqref{matrixK} and \eqref{matrixU},
and 
\begin{align*}
P_{\oplus}
 & =
K_{+\oplus}
(- A_{\oplus \oplus}^{-1})   \tilde{C}_{\oplus}(- A_{\oplus \oplus}^{-1}) 
(A_{\oplus-}
+
A_{\oplus+}\Psi)
\end{align*}
with $K_{+\oplus}={C_+^{-1}A_{+\oplus}}+\Psi |C_-^{-1}|A_{-\oplus}$.
\end{theorem}

\begin{proof}
To remove the effect of $\ve^{-1}$ in the left-most coefficient of
\eqref{RiccEqPsiNullPlus}, 
we pre-multiply both sides by $\diag(I,\ve I)$.
%
%
%
%
For
$\mathcal{X}=
\left[\begin{array}{c}
\mathcal{X}_{+-}
\\
\mathcal{X}_{\op-}
\end{array}\right]
$,  we define the operator
\begin{align*}
F\left(\varepsilon,\mathcal{X}
\right)
= &
\left[\begin{array}{c}
(C_{+}+\varepsilon\tilde{C}_{+})^{-1}
(A_{+-} + A_{++}\mathcal{X}_{+-}+A_{+\op}\mathcal{X}_{\op-})\\
\tilde{C}_{\oplus}^{-1}
(A_{\oplus-} + A_{\op+}\mathcal{X}_{+-} + A_{\op\op} \mathcal{X}_{\op-})
\end{array}\right]
\\
& +\left[\begin{array}{c}
\mathcal{X}_{+-}
\\
\varepsilon \mathcal{X}_{\oplus-}
\end{array}\right]
|C_{-}+\varepsilon\tilde{C}_{-}|^{-1}
(
A_{--}
+A_{-+}\mathcal{X}_{+-}
+ A_{-\oplus}\mathcal{X}_{\oplus-}
).
\end{align*}
 The equation 
\begin{equation*}
\left.{\partial}_{\mathcal{X}}F(\varepsilon,\mathcal{X})\right|_{\varepsilon =0,\mathcal{X}=\overline{\Psi}}
\left[\begin{array}{c}
Y_{+-}\\
Y_{\oplus-}
\end{array}\right]
=\left[\begin{array}{c}
H_{+-}\\
H_{\oplus-}
\end{array}\right]
\end{equation*}
is equivalent to the set of equations
\begin{align*}
Y_{+-}U+KY_{+-} & =  H_{+-}+K_{+\oplus}(-A_{\oplus\oplus}^{-1})\tilde{C}_{\oplus}H_{\oplus-},\\
Y_{\oplus-} & =  A_{\oplus\oplus}^{-1}\tilde{C}_{\oplus}H_{\oplus-}+(-A_{\oplus\oplus}^{-1})A_{\oplus-}Y_{+-}.
\end{align*}
This is a non-singular system, so that $\bs{\Psi}(\varepsilon)$ is analytic,
by the Implicit Function Theorem.
 From \eqref{RiccEqPsiNullPlus}, we obtain the two equations:
 \begin{align}
 &
{\Psi}_{+-}(\varepsilon)\vert C_{-}+\varepsilon\tilde{C}_{-}\vert^{-1}
\left(
A_{--}
+
A_{-+}{\Psi}_{+-}(\varepsilon) 
+ A_{-\oplus}
{\Psi}_{\oplus-}(\varepsilon)
\right)
\nonumber
\\
&+ 
(C_{+} +\varepsilon\tilde{C}_{+})^{-1}
\left(
A_{+-}
+
A_{++}{\Psi}_{+-}(\varepsilon)
+
A_{+\oplus}{\Psi}_{\oplus-}(\varepsilon)
\right)
 =  0,
 \label{eqtodeterminPsiplusrondmoins01}
 \\
 \intertext{and}
  &
  \varepsilon
{\Psi}_{\oplus-}(\varepsilon)
\vert C_{-}+\varepsilon\tilde{C}_{-}\vert^{-1}
\left(
A_{--}
+
A_{-+}{\Psi}_{+-}(\varepsilon) 
+ A_{-\oplus}
{\Psi}_{\oplus-}(\varepsilon)
\right)
\nonumber
\\
&+ 
\tilde{C}_{\oplus}^{-1}
\left(
A_{\oplus-}
+
A_{\oplus+}{\Psi}_{+-}(\varepsilon)
+
A_{\oplus\oplus}{\Psi}_{\oplus-}(\varepsilon)
\right)
 =  0,
 \label{eqtodeterminPsiplusrondmoins}
 \end{align}
 in which we take the limit for $\varepsilon\rightarrow0$. The second
equation gives
\begin{equation}
   \label{e:plusmoins}
{\Psi}_{\oplus-}(0)
=(-A_{\oplus\oplus})^{-1}
\left(
A_{\oplus-}+A_{\oplus+}{\Psi}_{+-}(0)
\right)
\end{equation}
and the first equation gives ${\Psi}_{+-}(0)$ as the
solution of \eqref{RiccEquaNullph}, so that ${\Psi}_{+-}(0)=\Psi.$
This proves \eqref{Psioverline}.

Taking the coefficients of $\ve$ in \eqref{eqtodeterminPsiplusrondmoins} and using \eqref{e:plusmoins}
 leads directly to \eqref{eqforPsoplusdsunrondm}.
To prove \eqref{eqforPsiprimefirstPert}, we note that
$\lim_{\ve\rightarrow 0}U(\ve)=U$ so that,  taking in 
(\ref{eqtodeterminPsiplusrondmoins01})
the limit for
$\varepsilon\rightarrow0$ and using (\ref{Psioverline}), 
we
obtain 
\begin{equation} \label{relationPsiUforoplus}
-\Psi U = C_+^{-1}(A_{+-}+A_{++}\Psi+A_{+\oplus}\Psi_{\oplus-}).
\end{equation}
We take the coefficient  of $\ve$ in \eqref{eqtodeterminPsiplusrondmoins01} and we use \eqref{relationPsiUforoplus} to obtain
\begin{equation*}
K_{++} \Psi^{(1)}_{+-}
+K_{+\oplus} \Psi^{(1)}_{\oplus-}
 +\Psi^{(1)}_{+-} U 
= 
-\Psi |C_{-}^{-1}| \tilde{C}_{-} U
-  C_{+}^{-1} \tilde{C}_{+} \Psi U
\end{equation*}
with
$K_{++}={C_+^{-1}A_{++}}+\Psi |C_-^{-1}|A_{-+}$.
Using \eqref{eqforPsoplusdsunrondm} and \eqref{matrixK} gives then \eqref{eqforPsiprimefirstPert}.
\end{proof}

\begin{rem}\rm{
  The components of the block $\Psi$ in $\overline{\Psi}$ are those
  defined in \eqref{PsiBasis}, for which one has a clear
  interpretation.  The components of the second block have a
  probabilistic interpretation as well: the $ij$th entry, for $i \in
  \mathcal{S}_{\oplus}$ and $j\in \mathcal{S}_{-}$, is the sum of
\begin{itemize}
\item 
$[(-A_{\oplus\oplus}^{-1})
A_{\oplus-}]_{ij}$,
the probability  that the phase process eventually goes from phase $i$ to phase $j$, after some time
spent in  $S_{\oplus}$  and
\item 
$[(-A_{\oplus\oplus}^{-1})
A_{\oplus+}\Psi]_{ij}$,
the  probability that the phase process leaves
  $\mathcal{S}_{\oplus}$  for a phase in $\mathcal{S}_+$ and
later returns to the initial level in phase $j$.
\end{itemize}
}
\end{rem}

\begin{rem}\rm{
The Sylvester equations \eqref{Psi1eqforNullNP} and 
\eqref{eqforPsiprimefirstPert} for $\Psi^{(1)}_{+-}$ are nearly
identical.  The only difference is the last term in the right-hand
side of  \eqref{eqforPsiprimefirstPert}, reflecting the migration of
all phases of $\mathcal{S}_0$ to phases of fluid increase.
}
\end{rem}


\subsection{Migration of $\mathcal{S}_0$ to $\mathcal{S}_-$}\label{Casemigrto-}

Assume that $\tilde c_i < 0$ for all $i$ in $\mathcal S_0$, so that
all the phases of $\mathcal{S}_{0}$ become phases of $\mathcal{S}_{-}$
after perturbation.  The set of such phases is written
$\mathcal{S}_{\ominus}$ and the infinitesimal generator of the phases
is written as
\[
A=\left[\begin{array}{ccc}
A_{++} & A_{+\ominus} & A_{+-}\\
A_{\ominus+} & A_{\ominus\ominus} & A_{\ominus-}\\
A_{-+} & A_{-\ominus} & A_{--}
\end{array}\right].
\]
The matrix of first return probabilities of the perturbed model is partitioned as 
\[
\bs{\Psi}(\varepsilon)=\left[\begin{array}{cc}
\boldsymbol{\Psi}_{+\ominus}(\varepsilon) & \boldsymbol{\Psi}_{+-}(\varepsilon)\end{array}\right],
\] 
and it is the minimal nonnegative solution of a  Riccati equation
which we rewrite as the two equations 
\begin{align}
 &(C_{+} + \varepsilon \tilde{C}_{+})^{-1}
(A_{+\ominus}+A_{++}{\Psi}_{+\ominus}(\varepsilon))
+
{\Psi}_{+\ominus}(\varepsilon)
\vert \varepsilon\tilde{C}_{\ominus}\vert^{-1}
(A_{\ominus\ominus}+A_{\ominus+}
{\Psi}_{+\ominus}(\varepsilon))
\nonumber
\\
&+
{\Psi}_{+-}(\varepsilon)
\vert C_{-}+\varepsilon\tilde{C}_{-}\vert^{-1}
(A_{-\ominus} + A_{-+}{\Psi}_{+\ominus}(\varepsilon))
=0, \quad 
\label{EqPsiPMoinsdsunRond}
\\
& (C_{+} + \varepsilon \tilde{C}_{+})^{-1}
(A_{+-}+A_{++}{\Psi}_{+-}(\varepsilon))
\nonumber
+
{\Psi}_{+\ominus}(\varepsilon)
\vert \varepsilon\tilde{C}_{\ominus}\vert^{-1}
(A_{\ominus-}+A_{\ominus+}
{\Psi}_{+-}(\varepsilon))
\nonumber
\\
& +
{\Psi}_{+-}(\varepsilon)
\vert C_{-}+\varepsilon\tilde{C}_{-}\vert^{-1}
(A_{--} + A_{-+}{\Psi}_{+-}(\varepsilon))
=0.
\label{EqPsiPMoins}
 \end{align}

\begin{theorem}\label{ThmFour}
The matrix $\bs{\Psi}(\varepsilon)$ of first return probabilities, minimal nonnegative solution to (\ref{EqPsiPMoinsdsunRond}) and (\ref{EqPsiPMoins})
 is near zero and may be written as
\begin{equation}
\bs{\Psi}(\varepsilon)=\overline{\Psi}+\ve\Psi^{(1)}+O(\ve^2),
\end{equation} 
where
\begin{align}\label{RefPsi0VersMinus}
\overline{\Psi}
&=\left[\begin{array}{cc}
0 & \Psi\end{array}\right],
\\
\Psi^{(1)}&=\left[\begin{array}{cc}
{\Psi}_{+\ominus}^{(1)}& {\Psi}_{+-}^{(1)}\end{array}\right].
\end{align}
The matrix  $\Psi$ is given in  \eqref{RiccEquaNullph},
${\Psi}_{+-}^{(1)}$ is  the unique solution of the Sylvester equation
\begin{align}\label{Sylveqforzerotominuscase}
K X +X U 
= 
-\Psi |C_{-}^{-1}| \tilde{C}_{-} U
-  C_{+}^{-1} \tilde{C}_{+} \Psi U
- KP_{\ominus}
\end{align}
and 
\begin{align}\label{eqforPsiplusminusronone}
\Psi^{(1)}_{+\ominus}
&=
(C_+^{-1} A_{+\ominus} + \Psi |C_-^{-1}| A_{-\ominus})\left(-A_{\ominus \ominus}^{-1}\right)
|\tilde{C}_{\ominus}|,
\end{align}
the matrices  $K$ and $U$ are defined in \eqref{matrixK} and \eqref{matrixU} 
and 
\begin{align*}
 P_{\ominus}
 & =
 \Psi^{(1)}_{+\ominus}
\left(-A_{\ominus \ominus}^{-1}\right)
(A_{\ominus-}+A_{\ominus+}\Psi).
\end{align*}
\end{theorem}

\begin{proof}
Here, to remove the effect of $\ve^{-1}$ as a coefficient of
$|\tilde{C}_{\om}^{-1}|$ in 
\eqref{EqPsiPMoinsdsunRond} and \eqref{EqPsiPMoins},
we define 
$\Gamma(\ve) = \ve^{-1}\Psi_{+\om}(\ve)$.
We define  the operator, for
$\mathcal{X}=
\left[\begin{array}{cc}
\mathcal{X}_{+\om}&
\mathcal{X}_{+-}
\end{array}\right]
$,
\begin{align*}
 F\left(\varepsilon,\mathcal{X}
\right)
&=
(C_{+}+\varepsilon\tilde{C}_{+})^{-1}
\left[\begin{array}{cc}
A_{+\om}+\ve A_{++} \mathcal{X}_{+\om}
&
A_{+-}+A_{++}\mathcal{X}_{+-}
\end{array}\right]
\\
& \hspace{0.3cm}
+\left[\begin{array}{cc}
\mathcal{X}_{+\om}|\tilde{C}_{\om}^{-1}|
&
 \mathcal{X}_{+-}|C_{-}+\varepsilon\tilde{C}_{-}|^{-1}
\end{array}\right]
\\
& \hspace{0.7cm} 
\times
\left[\begin{array}{cc}
A_{\om\om} +\ve A_{\om+}\mathcal{X}_{+\om}  &   A_{\om-}+A_{\om+}\mathcal{X}_{+-}
\\
A_{-\om} +\ve A_{-+}\mathcal{X}_{+\om}  &   A_{--}+A_{-+}\mathcal{X}_{+-}
\end{array}\right].
\end{align*}
One shows that $\left[\begin{array}{cc}
\Psi_{+\om}^{(1)}&
\Psi
\end{array}\right]$
is a solution of $F(\ve,\mathcal{X})=0$,
where 
$\Psi_{+\om}^{(1)}$ is defined in 
\eqref{eqforPsiplusminusronone}.
Next, we take the derivative of $F$ with respect to $\mathcal{X}$,
evaluated at $\ve=0$,
$\mathcal{X}=
\left[\begin{array}{cc}
\Psi_{+\om}^{(1)}&
\Psi
\end{array}\right]
$.
The system is equivalent to the set of equations
\begin{align*}
Y_{+-}U+KY_{+-} & =  H_{+-}+ H_{+\om} (-A_{\om\om}^{-1})(A_{\om-}+A_{\om+}\Psi),\\
  Y_{+\om}  &= Y_{+-} |C_-^{-1}| A_{-\om} (-A_{\om\om}^{-1})|\tilde{C}_{\om}|
+H_{+\om}A_{\om\om}^{-1}|\tilde{C}_{\om}|,
\end{align*}
where, by \eqref{EqRefQ}, \eqref{matrixU}, \eqref{matrixK},
\begin{align}
U & =  {\left|C_{-}^{-1}\right|(A_{--}}+A_{-\ominus}\left(-A_{\ominus\ominus}^{-1}\right)A_{\ominus-})
+{\left|C_{-}^{-1}\right|(A_{-+}}+A_{-\ominus}\left(-A_{\ominus\ominus}^{-1}\right)A_{\ominus+})\Psi,
\label{UforThm4}\\
K & =  C_{+}^{-1}(A_{++}+A_{+\ominus}\left(-A_{\ominus\ominus}^{-1}\right)A_{\ominus+})
+\Psi\left|C_{-}^{-1}\right|(A_{-+}+A_{-\ominus}\left(-A_{\ominus\ominus}^{-1}\right)A_{\ominus+}).
\label{KforThm4}
\end{align}
The system is non-singular so that 
$\left[\begin{array}{cc}
\Gamma(\ve)&
\Psi_{+-}(\ve)
\end{array}\right]$
 is analytic.

 The block components of $\overline{\Psi}$ are obtained as follows.
As
$\ve \Gamma(\ve) = \Psi_{+\om}(\ve)$,
we find that
$\Psi_{+\om}(0)=0$.
Next, define
\begin{align}\label{XlimPsidiveps}
W=\lim_{\varepsilon \rightarrow 0}
\Gamma(\ve)\vert \tilde{C}_{\ominus}\vert^{-1}
\end{align}
which is finite since $\Gamma(\ve)$ is analytic. 
We rewrite 
\eqref{EqPsiPMoinsdsunRond} and find that
\begin{align}\label{valueXPsiPM}
W & =  
C_{+}^{-1}A_{+\ominus}(-A_{\ominus\ominus})^{-1}
 +
\lim_{\varepsilon \rightarrow 0}{\Psi}_{+-}(\varepsilon)
\vert C_{-}\vert^{-1}
A_{-\ominus} 
(-A_{\ominus\ominus})^{-1}.
\end{align}
Taking the limit as $\varepsilon\rightarrow 0$
in \eqref{EqPsiPMoins}
and replacing $W$ by
\eqref{valueXPsiPM}
leads to \eqref{RiccEquaNullph}. Thus,
$\lim_{\varepsilon \rightarrow 0}{\Psi}_{+-}(\varepsilon)=\Psi,$
%
%
and  \eqref{RefPsi0VersMinus} is proved.

The block components of $\Psi^{(1)}$ are obtained as follows.
Taking the coefficients of $\ve^0$ in \eqref{EqPsiPMoinsdsunRond} gives directly \eqref{eqforPsiplusminusronone}.
To show \eqref{Sylveqforzerotominuscase}, 
we take the coefficients of $\ve^2$  in
\eqref{EqPsiPMoins}  and get the equation
\begin{align}
\Psi_{+\ominus}^{(2)}|\tilde{C}_{\ominus}^{-1}|(A_{\ominus-}+A_{\ominus+}\Psi)
= &
-(\Psi_{+-}^{(1)}|C_-^{-1}| 
+
 \Psi |C_{-}^{-1}| \tilde{C}_{-} |C_{-}^{-1}|
 )
 (A_{--}+A_{-+}\Psi)
 \nonumber
\\
& -
(C_+^{-1}A_{++} +\Psi |C_-^{-1}|A_{-+}+\Psi_{+\ominus}^{(1)} |\tilde{C}_{\ominus}^{-1}| A_{\ominus +})
\Psi_{+-}^{(1)} \nonumber 
\\
 &+ C_+^{-1}\tilde{C}_+ C_+^{-1}(A_{+-}+A_{++}\Psi).
\label{Psi2EqtwoThm3}
\end{align}
We equate the coefficients of $\ve$ in \eqref{EqPsiPMoinsdsunRond} and get
\begin{align}
\Psi_{+\ominus}^{(2)}|\tilde{C}_{\ominus}^{-1}|
= &
(
C_{+}^{-1} A_{++}+\Psi|C_-^{-1}|A_{-+}
+\Psi_{+\ominus}^{(1)}|\tilde{C}_{\ominus}^{-1}|A_{\ominus+}
)
\Psi_{+\ominus}^{(1)}
(-A_{\ominus\ominus}^{-1})
\nonumber
\\ & +
(
\Psi |C_{-}^{-1}|\tilde{C}_{-}
+\Psi_{+-}^{(1)}
)
|C_{-}^{-1}|A_{-\ominus}(-A_{\ominus\ominus}^{-1})
\nonumber
\\
& -C_+^{-1}\tilde{C}_+C_{+}^{-1}A_{+\ominus}(-A_{\ominus\ominus}^{-1}).
\label{Psi2exprThm4}
\end{align}
By the Riccati equation \eqref{RiccEquaNullph} and the definition \eqref{UforThm4} of $U$, we have 
\begin{align*}
-\Psi U&=C_{+}^{-1}(A_{+-}+A_{+\ominus}\left(-A_{\ominus\ominus}^{-1}\right)A_{\ominus-})
+C_{+}^{-1}\left(A_{++}+A_{+\ominus}\left(-A_{\ominus\ominus}^{-1}\right)A_{\ominus+}\right)\Psi.
\end{align*}
We replace the first coefficient $\Psi_{+\om}^{(1)}$ in \eqref{Psi2exprThm4} 
by its expression \eqref{eqforPsiplusminusronone},
then we replace 
$\Psi_{+\ominus}^{(2)}|\tilde{C}_{\ominus}^{-1}|$
 in \eqref{Psi2EqtwoThm3} by the modified right-hand side of \eqref{Psi2exprThm4}.
We put together the coefficients of $\Psi_{+-}^{(1)}$, use 
\eqref{UforThm4}, \eqref{KforThm4}
and eventually obtain \eqref{Sylveqforzerotominuscase}.
\end{proof}

\begin{rem}\rm{
   \label{explainPsiPlMin0is0}
   The physical justification of ${\Psi}_{+\ominus}(0)=0$ goes as
   follows: $({\Psi}_{+\ominus}(\varepsilon))_{ij}$ is the probability
   that the level moves to $0$ in phase
   $j\in \mathcal{S}_{\ominus},$ given that the initial level is $0$
   and the phase is $i\in \mathcal{S}_{+}$, in the limit, when
   $\varepsilon$ approaches $0,$ this probability tends to $0$ because the
   fluid can only return to level zero in a phase of $\mathcal{S}_{-}$.
   }
\end{rem}


\subsection{General case}\label{GenCaseMig}

Assume $\tilde{c}_i\neq 0$ for $i$ in $\mathcal{S}_0$, so that all the
phases of $\mathcal{S}_{0}$ disseminate in $\mathcal{S}_{+}$ and $\mathcal{S}_{-}$
after perturbation. 
The infinitesimal generator becomes
\begin{equation}\label{GeneratorGeneralCase}
A=\left[\begin{array}{cccc}
A_{++} & A_{+\oplus} & A_{+\ominus} & A_{+-}\\
A_{\oplus+} & A_{\oplus\oplus} & A_{\oplus\ominus} & A_{\oplus-}\\
A_{\ominus+} & A_{\ominus\oplus} & A_{\ominus\ominus} & A_{\ominus-}\\
A_{-+} & A_{-\oplus} & A_{-\ominus} & A_{--}
\end{array}\right].
\end{equation}
We find here a superposition of the effects observed in the two special cases examined in Sections \ref{AffectedCase+} and \ref{Casemigrto-}.
The matrix of first return probabilities from above of the perturbed
system  takes the form
\begin{equation}\label{PsiDivis4}
\boldsymbol{\Psi}(\varepsilon)=\left[\begin{array}{cc}
\Psi_{+\ominus}(\varepsilon) & \Psi_{+-}(\varepsilon)\\
\Psi_{\oplus\ominus}(\varepsilon) & \Psi_{\oplus-}(\varepsilon)
\end{array}\right],
\end{equation}
it is the unique solution of the usual Riccati equation which  may be
rewritten as the following  set of four equations:

{\footnotesize 
\begin{align}
&
C_{+}^{-1}(\varepsilon) A_{+\ominus}
+C_{+}^{-1}(\varepsilon)\vligne{A_{++} & A_{+\oplus}}
         \vligne{\Psi_{+\ominus}(\ve) \\  \Psi_{\op\om}(\ve)}
+\vligne{\Psi_{+\om}(\ve) &\Psi_{+-}(\ve) }
   \vligne{U_{\om\om}(\ve) \\U_{-\om}(\ve) }
=
0,
 \label{A_11}
\\
&
    C_{+}^{-1}(\varepsilon) A_{+-}
 + C_{+}^{-1}(\varepsilon)\vligne{A_{++} & A_{+\oplus}}
         \vligne{\Psi_{+-}(\ve) \\  \Psi_{\op-}(\ve)}
+ \vligne{\Psi_{+\om}(\ve) &\Psi_{+-}(\ve) }
   \vligne{U_{\om-}(\ve) \\U_{--}(\ve) }
=
0,
\label{A_12}
\\
&
    (\ve \tilde{C}_{\op})^{-1} A_{\op\om}
 +   (\ve \tilde{C}_{\op})^{-1}\vligne{A_{\op+} & A_{\op\oplus}}
        \vligne{\Psi_{+\ominus}(\ve) \\  \Psi_{\op\om}(\ve)}
+ \vligne{\Psi_{\op\om}(\ve) &\Psi_{\op-}(\ve) }
 \vligne{U_{\om\om}(\ve) \\U_{-\om}(\ve) }
=
0,
\label{A_21}
\\
&
    (\ve \tilde{C}_{\op})^{-1} A_{\op\om}
 +   (\ve \tilde{C}_{\op})^{-1}\vligne{A_{\op+} & A_{\op\oplus}}
        \vligne{\Psi_{+-}(\ve) \\  \Psi_{\op-}(\ve)}
+ \vligne{\Psi_{\op\om}(\ve) &\Psi_{\op-}(\ve) }
 \vligne{U_{\om-}(\ve) \\U_{--}(\ve) }
=
0.
\label{A_22}
\end{align}
}

We show below that $\bs{\Psi}(\ve)$ is analytic, thus we may write the matrices $U(\ve)$ and $K(\ve)$ as 
\begin{align}
U(\ve)&= \sum_{n=-1}^{\infty} \ve^n U_n
\text{\quad with \quad }
U_n = 
\left[\begin{array}{cc}
U_{\om\om}^{(n)} & U_{\om-}^{(n)} \\
U_{-\om}^{(n)} & U_{--}^{(n)} 
\end{array}\right],
\\
K(\ve)&= \sum_{n=-1}^{\infty} \ve^n K_n
\text{\quad with \quad }
K_n = 
\left[\begin{array}{cc}
K_{++}^{(n)} & K_{+\op}^{(n)} \\
K_{\op +}^{(n)} & K_{\op\op}^{(n)} 
\end{array}\right],
\end{align}
in particular, the  blocks
\begin{align}
U_{\om\om}^{(-1)}&=|\tilde{C}_{\om}^{-1}|A_{\om\om} + |\tilde{C}_{\om}^{-1}| A_{\om\op}\Psi_{\op\om},
\label{defUpertub}\\
K_{\op\op}^{(-1)}&=\tilde{C}_{\op}^{-1}A_{\op\op} + \Psi_{\op\om}|\tilde{C}_{\om}^{-1}|A_{\om\op},
\label{defKpertub}
\end{align}
play an important role in what follows.

\begin{theorem}\label{ThmFive}
The matrix $\bs{\Psi}(\varepsilon)$ of first return probabilities, minimal nonnegative solution to (\ref{A_11}-\ref{A_22}) 
 for the perturbed model is near zero and may be written as
\begin{equation*}
\bs{\Psi}(\varepsilon)=\overline{\Psi}+\ve\Psi^{(1)}+O(\ve^2),
\end{equation*} 
where
\begin{equation}
\overline{\Psi}=\left[\begin{array}{cc}
0 & \Psi\\
\Psi_{\oplus\ominus} & \Psi_{\oplus-}
\end{array}\right].
\end{equation}
The block $\Psi$ is  given in  \eqref{RiccEquaNullph},
\begin{align}\label{Psiprmgen}
\Psi_{\op -}
= (-K_{\op\op}^{(-1)})^{-1}
  &\big(
  			\tilde{C}^{-1}_{\op}(A_{\op-} +A_{\op+}\Psi) 
  			+
  			\Psi_{\op\om} |\tilde{C}_{\om}^{-1}|
  			    (A_{\om-}+A_{\om+}\Psi)
  	\big),
\end{align}
$\Psi_{\op\om}$ is the minimal nonnegative solution to the Riccati equation 
\begin{equation}\label{RiccatieqGencaseperturb}
\tilde{C}_{\op}^{-1}A_{\op\om}+\tilde{C}_{\op}^{-1}A_{\op\op}X+X|\tilde{C}_{\om}^{-1}|A_{\om\om}+X|\tilde{C}_{\om}^{-1}|A_{\om\op}X=0.
\end{equation}
Furthermore,
\begin{equation}
{\Psi}^{(1)}=\left[\begin{array}{cc}
\Psi^{(1)}_{+\ominus} & \Psi^{(1)}_{+-}\\
\Psi_{\oplus\ominus}^{(1)} & \Psi_{\oplus-}^{(1)}
\end{array}\right],
\end{equation}
with
\begin{align}\label{eqforPsionepom}
\Psi^{(1)}_{+\ominus}
&=
\big(
C_+^{-1} (A_{+\ominus}+A_{+\oplus}\Psi_{\op\om}) 
+
 \Psi |C_-^{-1}| + A_{-\op}\Psi_{\op\om})
\big)
(-U_{\ominus \ominus}^{(-1)})^{-1},
\end{align}
${\Psi}_{\op\om}^{(1)}$ is the unique solution of the Sylvester equation
\begin{align}\label{SylvEqforPsiopom}
K_{\op\op}^{(-1)} X +X U_{\om\om}^{(-1)} 
= &
-(\tilde{C}^{-1}_{\op}A_{\op+}+\Psi_{\op\om}|C_{\om}^{-1}|A_{\om+})\Psi_{+\om}^{(1)}
\nonumber\\
&-
\Psi_{\op-}|C_-^{-1}|(A_{-\om}+A_{-+}\Psi_{+\om}+A_{-\op}\Psi_{\op\om}),
\end{align}
and with
\begin{align}\label{eq:psiopmderivgen}
\Psi_{\op-}^{(1)}
&=
(-K_{\op\op}^{(-1)})^{-1}
\big(
	K_{\op+}^{(-1)}\Psi_{+-}^{(1)}
	+ \Psi_{\op\om}|C_-^{-1}|U_{\om-}^{(-1)}
	- \Psi_{\op -} U_{--}^{(0)}
	\big),
\end{align}
and
$\Psi_{+-}^{(1)}$ is the unique solution of the Sylvester equation
\begin{align}\label{Psi1finalgen}
& (C_+^{-1} A_{++} + \Psi |C_-^{-1}|A_{-+}) \Psi_{+-}^{(1)}+\Psi_{+-}^{(1)} U_{--}^{(0)}
\nonumber \\ &\hspace{0.5cm}+
(C_+^{-1} A_{+\op} + \Psi |C_-^{-1}|A_{-\op}) \Psi_{+\op}^{(1)}
+\Psi_{+\om}^{(2)}U_{\om-}^{(-1)}
\nonumber \\
&\hspace{0.5cm}=
C_+^{-1} \tilde{C_+}C_+^{-1} (A_{+-}+A_{++}\Psi+A_{+\op}\Psi_{\op-})
-\Psi |C_-^{-1}|\tilde{C}_-  U_{--}^{(0)},
\end{align}
where
\begin{align}\label{eqforPsitwopom}
\Psi^{(2)}_{+\ominus}
=&
\big(
-C_+^{-1}C_+ C_+^{-1} (A_{+\ominus}+A_{+\oplus}\Psi_{\op\om}) 
+ C_+^{-1} (A_{++} \Psi_{+\om}^{(1)} + A_{+\op}\Psi_{\op\om}^{(1)})
\nonumber\\
  			&+
  (\Psi^{(1)}+\Psi |C_-^{-1}|\tilde{C}_-) U_{-\om}^{(0)}
  +
 \Psi |C_-^{-1}| (A_{-+}\Psi_{+\om}^{(1)}+ A_{-\op}\Psi_{\op\om}^{(1)})
\big)
(-U_{\ominus \ominus}^{(0)})^{-1}.
\end{align}
\end{theorem}

\begin{proof}

To remove the effect of $\ve^{-1}$ as $\ve \rightarrow 0$,
we need to combine the transformations of the previous two theorems.
We pre-multiply the Riccati equation by $\diag(I,\ve I)$
and we use the matrix $\Gamma(\ve) = \ve^{-1}\Psi_{+\om}(\ve)$.
We obtain a new fixed-point equation, from which we eventually prove,
by following the same steps as in Theorem \ref{Theorem32} and Theorem \ref{ThmFour}, 
that the solutions are matrices of analytic functions.

Observe the terms in $\ve^{-1}$ in the equations (\ref{A_11}) to (\ref{A_22}):
\begin{itemize}
\item 
we conclude from 
 \eqref{A_11} that $\Psi_{+\om}=0$ by a similar argument to the proof of Theorem \ref{ThmFour};
 \item
multiply  \eqref{A_21} by $\ve$ and let $\ve$  tend to zero to obtain the Riccati equation \eqref{RiccatieqGencaseperturb} satisfied by 
  $\Psi_{\op\om}$;
 \item
multiply  \eqref{A_22} by $\ve$  and let $\ve$   tend to zero, 
 gives  \eqref{Psiprmgen}, 
taking into  account that
 $\lim_{\ve \rightarrow 0} \Psi(\ve)=\Psi$, an equality that is proved below.
\end{itemize}
 To determine $\Psi_{+-}(0)$ is more involved.
 We proceed as follows. 
 First, from \eqref{A_11}, we take the terms in $\ve^{0}$ and we find the expression 
 \eqref{eqforPsionepom} for $\Psi^{(1)}_{+\ominus}$ that we replace in \eqref{A_12}.
 From \eqref{A_12}, we take the terms in $\varepsilon^{0}$ and obtain $\Psi_{+-}$, after a reorganization of the terms,
as the minimal nonnegative solution to 
\begin{equation*}
C_{+}^{-1}T_{+-}+C_{+}^{-1}T_{++}X+X|C_{-}^{-1}|T_{--}+X|C_{-}^{-1}|T_{-+}X=0,
\end{equation*}
with
\begin{align*}
\left[\begin{array}{cc}
T_{++} & T_{+-}\\
T_{-+} & T_{--}
\end{array}\right]
=& \left[\begin{array}{cc}
A_{++} & A_{+-}\\
A_{-+} & A_{--}
\end{array}\right]
\nonumber
\\
 &+
\left[\begin{array}{cc}
A_{+\oplus} & A_{+\ominus}\\
A_{-\oplus} & A_{-\ominus}
\end{array}\right]
\left[\begin{array}{cc}
D_{\oplus\oplus} & D_{\oplus\ominus}\\
D_{\ominus\oplus} & D_{\ominus\ominus}
\end{array}\right]
\left[\begin{array}{cc}
A_{\oplus+} & A_{\oplus-}\\
A_{\ominus+} & A_{\ominus-}
\end{array}\right].
\end{align*}
where
\begin{align*}
D_{\oplus\oplus} & =  (-K_{\oplus\oplus}^{(-1)})^{-1}\tilde{C}_{\oplus}^{-1}+\Psi_{\oplus\ominus}D_{\ominus\oplus},\\
D_{\ominus\oplus} & =  (-U_{\ominus\ominus}^{(-1)})^{-1}|\tilde{C}_{\ominus}^{-1}|A_{\ominus\oplus}(-K_{\oplus\oplus}^{(-1)})^{-1}\tilde{C}_{\oplus}^{-1},\\
D_{\oplus\ominus} & =  (-K_{\oplus\oplus}^{(-1)})^{-1}\Psi_{\oplus\ominus}|\tilde{C}_{\ominus}^{-1}|+\Psi_{\oplus\ominus}D_{\ominus\ominus},\\
D_{\ominus\ominus} & =  (-U_{\ominus\ominus}^{(-1)})^{-1}|\tilde{C}_{\ominus}^{-1}|(I+A_{\ominus\oplus}(-K_{\oplus\oplus}^{(-1)})^{-1}\Psi_{\oplus\ominus}|\tilde{C}_{\ominus}^{-1}|).
\end{align*}
To prove that the matrix $T$ is identical to the matrix $Q$ defined in
\eqref{EqRefQ}, we only need to show that the matrix made up of the
four blocks labeled with $D$s
is equal to
$(-A_{00}^{-1})$, partitionned according to \eqref{GeneratorGeneralCase}, as
\begin{equation}
(-A_{00}^{-1})
=
\left[\begin{array}{cc}
B_{\oplus\oplus} & B_{\oplus\ominus}\\
B_{\ominus\oplus} & B_{\ominus\ominus}
\end{array}\right]
\end{equation} 
where
\begin{align*}
B_{\oplus\oplus} & =  -(A_{\oplus\oplus}+A_{\oplus\ominus}(-A_{\ominus\ominus}^{-1})A_{\ominus\oplus})^{-1},\\
B_{\ominus\oplus} & =  (-A_{\ominus\ominus}^{-1})A_{\ominus\oplus}B_{\oplus\oplus},\\
B_{\oplus\ominus} & =  B_{\oplus\oplus}A_{\oplus\ominus}(-A_{\ominus\ominus}^{-1}),\\
B_{\ominus\ominus} & =  -(A_{\ominus\ominus}+A_{\ominus\oplus}(-A_{\oplus\oplus}^{-1})A_{\oplus\ominus})^{-1}.
\end{align*}

By \eqref{RiccatieqGencaseperturb}, we have 
\begin{equation*}
A_{\oplus\ominus}=-\tilde{C}_{\oplus}K_{\oplus\oplus}^{(-1)}\Psi_{\oplus\ominus}-\tilde{C}_{\oplus}\Psi_{\oplus\ominus}|\tilde{C}_{\ominus}^{-1}|A_{\ominus\ominus}
\end{equation*}
so that 
\begin{align*}
B_{\oplus\oplus} 
  &= -\big(A_{\oplus\oplus}
  -\tilde{C}_{\oplus}K_{\oplus\oplus}^{(-1)}\Psi_{\oplus\ominus}(-A_{\ominus\ominus}^{-1})A_{\ominus\oplus}
   +  \tilde{C}_{\oplus}\Psi_{\oplus\ominus}|\tilde{C}_{\ominus}^{-1}|A_{\ominus\oplus}\big)^{-1}
  \\
&  =  (I-\Psi_{\oplus\ominus}(-A_{\ominus\ominus}^{-1})A_{\ominus\oplus})^{-1}(-K_{\oplus\oplus}^{(-1)})^{-1}\tilde{C}_{\oplus}^{-1},
\end{align*}
using \eqref{defKpertub}. We write 
\begin{align*}
(I-\Psi_{\oplus\ominus}(-A_{\ominus\ominus}^{-1})A_{\ominus\oplus})^{-1} 
&  =I+\Psi_{\oplus\ominus}(I-(-A_{\ominus\ominus}^{-1})A_{\ominus\oplus}\Psi_{\oplus\ominus})^{-1}(-A_{\ominus\ominus}^{-1})A_{\ominus\oplus}\\
 & =  I+\Psi_{\oplus\ominus}(-U_{\ominus\ominus}^{(-1)}))^{-1}|\tilde{C}_{\ominus}^{-1}|A_{\ominus\oplus},
\end{align*}
so that $B_{\oplus\oplus}=D_{\oplus\oplus}$.

Next, we have 
\begin{align*}
B_{\ominus\oplus} 
 & =  (-A_{\ominus\ominus}^{-1})A_{\ominus\oplus}
 (I+A_{\ominus\oplus}\Psi_{\oplus\ominus}(-U_{\ominus\ominus}^{(-1)})^{-1}
 |\tilde{C}_{\ominus}^{-1}|)
 (-K_{\oplus\oplus}^{(-1)})^{-1}\tilde{C}_{\oplus}^{-1}\\
 & =  (-A_{\ominus\ominus}^{-1})
(-|\tilde{C}_{\ominus}|U_{\ominus\ominus}^{(-1)})+A_{\ominus\oplus}\Psi_{\oplus\ominus})
 (-U_{\ominus\ominus}^{(-1)})^{-1})
 |\tilde{C}_{\ominus}^{-1}|
 A_{\ominus\oplus}(-K_{\oplus\oplus}^{(-1)})^{-1}\tilde{C}_{\oplus}^{-1}.
\end{align*}
By \eqref{defUpertub}, 
$-|\tilde{C}_{\ominus}|U_{\ominus\ominus}^{(-1)}+A_{\ominus\oplus}\Psi_{\oplus\ominus}$
simplifies to $-A_{\om\om}$
so that 
$B_{\ominus\oplus} =  D_{\ominus\oplus} $.
%
%

Then, we have
 \begin{align*}
B_{\oplus\ominus} 
  = & (-K_{\oplus\oplus}^{(-1)})^{-1}\tilde{C}_{\oplus}^{-1}A_{\oplus\ominus}(A_{\ominus\ominus}^{-1})
 \\
 &  +\Psi_{\oplus\ominus}(-U_{\ominus\ominus}^{(-1)})^{-1}|\tilde{C}_{\ominus}^{-1}|A_{\ominus\oplus}(-K_{\oplus\oplus}^{(-1)})^{-1}\tilde{C}_{\oplus}^{-1}A_{\oplus\ominus}(A_{\ominus\ominus}^{-1}),
\end{align*}
and we use \eqref {RiccatieqGencaseperturb} to  replace $\tilde{C}_{\oplus}^{-1}A_{\oplus\ominus}$
 in the first term to write
\begin{align*}
B_{\oplus\ominus}
  = & (-K_{\oplus\oplus}^{(-1)})^{-1}
 (-\Psi_{\oplus\ominus}|\tilde{C}_{\ominus}^{-1}|A_{\om\om}-K_{\op\op}^{(-1)}\Psi_{\oplus\ominus})
 (-A_{\ominus\ominus}^{-1})
\\ 
 &+\Psi_{\oplus\ominus}(-U_{\ominus\ominus}^{(-1)})^{-1}|\tilde{C}_{\ominus}^{-1}|A_{\ominus\oplus}(-K_{\oplus\oplus}^{(-1)})^{-1}\tilde{C}_{\oplus}^{-1}A_{\oplus\ominus}(-A_{\ominus\ominus}^{-1})\\
  = & (-K_{\oplus\oplus}^{(-1)})^{-1}\Psi_{\oplus\ominus}|\tilde{C}_{\ominus}^{-1}|
 \\&+
 \Psi_{\oplus\ominus}
 \big(I+
 (-U_{\ominus\ominus}^{(-1)})^{-1}|\tilde{C}_{\ominus}^{-1}|A_{\ominus\oplus}(-K_{\oplus\oplus}^{(-1)})^{-1}\tilde{C}_{\oplus}^{-1}A_{\oplus\ominus}
 \big)
 (-A_{\ominus\ominus}^{-1}).
\end{align*}
We use \eqref{defUpertub}, to write the second term as
\begin{align*}
 & \Psi_{\oplus\ominus}(-U_{\ominus\ominus}^{(-1)})^{-1}
 |\tilde{C}_{\ominus}^{-1}|
 \big(A_{\ominus\oplus}(-K_{\oplus\oplus}^{(-1)})^{-1}\tilde{C}_{\oplus}^{-1}A_{\oplus\ominus}
 +(-A_{\ominus\ominus}-A_{\ominus\oplus}\Psi_{\oplus\ominus})\big)
 (-A_{\ominus\ominus}^{-1})
 \\
 &=\Psi_{\oplus\ominus}(-U_{\ominus\ominus}^{(-1)})^{-1}
 |\tilde{C}_{\ominus}^{-1}| 
  \big( I-A_{\ominus\oplus}(-K_{\oplus\oplus}^{(-1)})^{-1}
  (\tilde{C}_{\oplus}^{-1}A_{\oplus\ominus}-K_{\oplus\oplus}^{(-1)}\Psi_{\oplus\ominus})
  (-A_{\ominus\ominus}^{-1})
  \big)\\
   &=\Psi_{\oplus\ominus}(-U_{\ominus\ominus}^{(-1)})^{-1}
 |\tilde{C}_{\ominus}^{-1}| 
\big( I-A_{\ominus\oplus}(-K_{\oplus\oplus}^{(-1)})^{-1}\Psi_{\oplus\ominus}|\tilde{C}_{\ominus}|  \big),
\end{align*}
were we used \eqref{RiccatieqGencaseperturb} to replace $K_{\oplus\oplus}^{(-1)}\Psi_{\oplus\ominus}$.
We find thus $B_{\op\om}=D_{\op\om}$.

Finally, we use the definition of $U_{\om\om}^{(-1)}$ to write 
\begin{align*}
B_{\ominus\ominus}
 & =  -(A_{\ominus\ominus}+A_{\ominus\oplus}\Psi_{\oplus\ominus}-A_{\ominus\oplus}(-A_{\oplus\oplus}^{-1})\tilde{C}_{\oplus}\Psi_{\oplus\ominus}U_{\ominus\ominus}^{(-1)})^{-1}\\
 & =  (-U_{\ominus\ominus}^{(-1)})^{-1}|\tilde{C}_{\ominus}^{-1}|(I-A_{\ominus\oplus}(-A_{\oplus\oplus}^{-1})\tilde{C}_{\oplus}\Psi_{\oplus\ominus}|\tilde{C}_{\ominus}^{-1}|)^{-1}.
\end{align*}
We write
\begin{align*}
&(I-A_{\ominus\oplus}(-A_{\oplus\oplus}^{-1})\tilde{C}_{\oplus}\Psi_{\oplus\ominus}|\tilde{C}_{\ominus}^{-1}|)^{-1}
\\
&\qquad\qquad = 
I + A_{\om\op}
\big(I-(-A_{\oplus\oplus}^{-1})\tilde{C}_{\oplus}\Psi_{\oplus\ominus}|\tilde{C}_{\ominus}^{-1}|A_{\om\op}\big)^{-1}
(-A_{\oplus\oplus}^{-1})\tilde{C}_{\oplus}\Psi_{\oplus\ominus}|\tilde{C}_{\ominus}^{-1}|
\\
&\qquad\qquad = 
I + A_{\om\op}
\big(-A_{\oplus\oplus}-\tilde{C}_{\oplus}\Psi_{\oplus\ominus}|\tilde{C}_{\ominus}^{-1}|A_{\om\op}\big)^{-1}
\tilde{C}_{\oplus}\Psi_{\oplus\ominus}|\tilde{C}_{\ominus}^{-1}|
\\
&\qquad\qquad = 
I + A_{\om\op}(-K_{\op\op}^{(-1)})^{-1}\Psi_{\oplus\ominus}|\tilde{C}_{\ominus}^{-1}|
\end{align*}
by \eqref{defKpertub}, so that $B_{\om\om}=D_{\om\om}$.

We find the block $\Psi_{\op-}^{(1)}$ of $\Psi^{(1)}$ given in \eqref{eq:psiopmderivgen}  by observing the terms in $\ve^{0}$ in
\eqref{A_22}.
From \eqref{A_21}, 
 we obtain the Sylvester  equation \eqref{eq:psiopmderivgen} for  $\Psi_{\op\om}^{(1)}$.
 Taking the terms in $\ve$ in \eqref{A_11} and \eqref{A_21} leads respectively to
 \eqref{Psi1finalgen} and \eqref{eqforPsitwopom}.
\end{proof}

\begin{rem}\rm{
Not surprisingly, 
$\Psi_{+\om}=0$, as  we found in \eqref{RefPsi0VersMinus}.

As in Section \ref{AffectedCase+},  \eqref{Psiprmgen} is a function of $\Psi$ but also of the supplementary component $\Psi_{\op\om}$.
This generalizes $\Psi_{\op-}$ given in \eqref{Psioverline}.
There is  a probabilistic interpretation similar to the one given in \eqref{Psioverline}, with, here, a correction term due to the introduction  of $\mathcal{S}_{\ominus}$:
$[\Psi_{\oplus-}]_{ij}$ is the sum of
\begin{itemize}
\item
$[(-K_{\op\op}^{(-1)})^{-1} \tilde{C}^{-1}_{\op}A_{\op-}]_{ij} $,
 the probability that the phase process goes from $i$ to $j$, after some time spent in phases of 
$\mathcal{S}_{\oplus}$ or $\mathcal{S}_{\ominus}$,
\item
$[(-K_{\op\op}^{(-1)})^{-1}\tilde{C}^{-1}_{\op}A_{\op+}\Psi ]_{ij}  $,
the probability that the process leaves $i$ for a phase in $\mathcal{S}_{+}$ and later returns to the initial level in $j$,
\item
$[(-K_{\op\op}^{(-1)})^{-1}\Psi_{\op\om} |\tilde{C}_{\om}^{-1}|A_{\om-}]_{ij} $
the probability that the process  comes back to the initial level in a phase of  $\mathcal{S}_{\ominus}$ and goes to $j$,
\item
$[(-K_{\op\op}^{(-1)})^{-1}\Psi_{\op\om} |\tilde{C}_{\om}^{-1}|A_{\om+}\Psi]_{ij} $
 the process  comes back to the initial level in a phase of  $\mathcal{S}_{\ominus}$, goes to a phase of $\mathcal{S}_{+}$
and later returns to the initial level in $j$,
\end{itemize}
for $i \in \mathcal{S}_{\oplus}$,
$j\in \mathcal{S}_{-}$.
}
\end{rem}

\begin{rem}\rm{
Higher order terms (in particular, the coefficients of $\ve^2$)
may be of interest in some cases.
It is clear that the principal difficulty lies in the necessity to deal with calculations that are steadily more cumbersome,
but no more.
We expect that coefficients of $\Psi_{+\ominus}$ or $\Psi_{\op-}$ will be given explicitly and that each successive coefficients of $\Psi_{+-}$ and $\Psi_{\op\om}$ will be solutions of Sylvester equations.
}
\end{rem}

\section{Impact on the stationary probability}\label{application}
For $j\in\mathcal{S}$ and $x\in\mathbb{R}^{+}$,
we define the joint distribution function of the level and the phase at time $t$, 
$F_{j}(x,t)=\mathbb{P}\left[X\left(t\right)\leq x,\varphi\left(t\right)=j\right],$
and its density by
\begin{equation*}
f_{j}(x,t)=\frac{\partial}{\partial x} F_{j}(x,t),
\text{\quad with \quad}
f_{j}(0,t)=\lim_{x\rightarrow 0}f_{j}(x,t).
\end{equation*}
%
%
The stationary density vector 
$\vpi (x)=(\pi_j(x):j\in\mathcal{S})$ of the fluid model, where, for $j\in\mathcal{S}$,
$\pi_j(x)=\lim_{t\rightarrow \infty}f_{j}(x,t),$
exists if and only if the mean stationary drift is negative, that is, if and only if 
$\sum_{i\in\mathcal{S}}\xi_i c_i <0$, where $\xi_i$ is defined in \eqref{defdeXistatphase} for all $i$. 
When the mean stationary drift of the fluid model is negative,
from Govorun {\it et al.} \cite{govorun2013stability}, we have, for $x>0$,
\begin{equation}\label{ExprePifluidq}
\boldsymbol{\pi}\left(x\right)=
\vc q
e^{Kx}
\left[\begin{array}{cc}
C_{+}^{-1}
\; ; 
\Psi\left\vert C_{-}\right\vert ^{-1}
\; ;
\Theta
\end{array}\right],
\end{equation}
and the mass at zero is $[0 \; ;  \vc p_- \; ;  \vc p_0]$
where 
\begin{align}
K&=C_{+}^{-1}Q_{++}+\Psi\left\vert C_{-}\right\vert ^{-1}Q_{-+},
\label{DefofKmatrix}
\\
\Theta &= 
\left(C_{+}^{-1}A_{+0}+\Psi\left\vert C_{-}\right\vert ^{-1}A_{-0}\right)\left(A_{00}\right)^{-1},
\label{DefofThmatrix}
\\
\vc q &= \vc p_- A_{-+} + \vc p_{0}A_{0+}
\label{Defofqvector}
\end{align}
and  $[\boldsymbol{p}_{-}\; ; \vc p_0]$
is the unique solution of the system
\begin{align}
\begin{bmatrix}
\vc p_- \; ; \vc p_0
\end{bmatrix}
\begin{bmatrix}
A_{--} + A_{-+}  \Psi & A_{-0}
\\
A_{0-}+A_{0+}\Psi & A_{00} 
\end{bmatrix}
&=
\vc 0
\label{solutionp0pm}
\\
[\boldsymbol{p}_{-} \; ; \vc p_0]\vone
+\boldsymbol{q}_{-}(-K)^{-1}
(C_{+}^{-1} + \Psi\left\vert C_{-}\right\vert ^{-1}+\Theta )\vc 1
&=1.
\label{normilzep}
\end{align}
 Expression  \eqref{ExprePifluidq} is numerically stable and has a physical interpretation (da Silva Soares [38, Chapter 1, Section 1.3]).
Furthermore, it appears clearly that all the quantities appearing in the expression of the stationary density are functions of~$\Psi$. 

 The stationary density of \eqref{TransMatrPertFluid} may be formulated as
\begin{equation}\label{piespil}
\boldsymbol{\pi}\left(x,\varepsilon\right)=
\vc q(\ve)
e^{K(\ve )x}
\left[\begin{array}{cc}
C_{+}^{-1}
\; ; 
\Psi(\ve) \left\vert C_{-}\right\vert ^{-1}
\; ;
\Theta(\ve)
\end{array}\right],
\end{equation} 
 where $K(\ve)$, $\Theta(\ve)$ and $\vc q(\ve)$ are defined similary to 
 \eqref{DefofKmatrix},\eqref{DefofThmatrix} and \eqref{Defofqvector} respectively.
It is well known that  the stationary density vector $\boldsymbol{\pi}(x, \varepsilon)$ is differentiable (see Kato \cite[Section 2]{kato2013perturbation}) 
and such that $\boldsymbol{\pi}(x, \varepsilon)$ may be  written as 
\begin{equation*}
\boldsymbol{\pi}(x,\varepsilon)=
 \boldsymbol{\pi}(x)+ \varepsilon \boldsymbol{\pi}^{(1)}(x,0) + O(\varepsilon^2),
\end{equation*}
where 
\begin{equation} \label{pi1pourTPerturb}
\boldsymbol{\pi}^{(1)}(x,0)=
\lim_{\varepsilon \rightarrow 0}\frac{\boldsymbol{\pi}(x,\varepsilon) - \boldsymbol{\pi}(x,0)}{\varepsilon},
\end{equation}
for all $x \in \mathbb{R}^{+}$. 
We find 
\begin{align*}
\boldsymbol{\pi}^{(1)}(x,0)
&=
\boldsymbol{q} e^{Kx}
\left[\begin{array}{ccc}
0
 \;  ; \;
\Psi^{(1)}\left\vert C_{-}^{-1}\right\vert
\;  ; \;
\Theta^{(1)}
\end{array}\right]
\nonumber
\\
& \hspace{-0.7cm}+ 
(\boldsymbol{q}^{(1)} e^{Kx}
+  \boldsymbol{q} L^{(1)}(x))
\left[\begin{array}{cc}
C_{+}^{-1}
\; ; 
\Psi\left\vert C_{-}\right\vert ^{-1}
\; ;
\Theta
\end{array}\right],
\end{align*}
where
$\Psi^{(1)}$ is given in Theorem \ref{ThmforAPerturbed}
and 
\begin{align*}
\Theta^{(1)}
&=
(C_+^{-1}A_{+0}
+ \Psi |C_-^{-1}|A_{-0})
(-A_{00}^{-1})\tilde{A_{00}}A_{00}^{-1}
\\
&\quad +
C_{+}^{-1}\tilde{A}_{+0}
+ \Psi^{(1)} |C_-^{-1}|A_{-0}
+\Psi |C_-^{-1}| \tilde{A}_{-0}.
\end{align*}
The vector $\vc q (\ve)$ is differentiable by Kato \cite[Section 2]{kato2013perturbation} and
\begin{equation*}
\boldsymbol{q}^{(1)}
= \vc p_-^{(1)} A_{-+} + \vc p_0^{(1)} A_{0+} 
+\vc p_{-}\tilde{A}_{-+} + \vc p_{0}^{(1)}\tilde{A}_{0+}
\end{equation*} 
with
\begin{align}\label{ppp0}
\begin{bmatrix}
\vc p_- ^{(1)}\; ; \vc p_0^{(1)}
\end{bmatrix}
&=
-
\begin{bmatrix}
\vc p_- \; ; \vc p_0
\end{bmatrix}
\begin{bmatrix}
\tilde{A}_{--} + \tilde{A}_{-+}  \Psi + A_{-+}\Psi^{(1)} & \tilde{A}_{-0}
\\
\tilde{A}_{0-}+\tilde{A}_{0+}\Psi +  A_{0+}\Psi^{(1)} & \tilde{A}_{00} 
\end{bmatrix}
\nonumber
\\
&
\begin{bmatrix}
A_{--} + A_{-+}  \Psi & A_{-0}
\\
A_{0-}+A_{0+}\Psi & A_{00} 
\end{bmatrix}^{\#}
+c \boldsymbol{\pi}(x),
\end{align}
where $M^{\#}$ denotes the group inverse of the matrix $M$.
We  find \eqref{ppp0} by solving the Poisson equation (see Meyer \cite{meyer1975role}) satsified by 
$[
\vc p_- ^{(1)}\; ; \vc p_0^{(1)}
]$,
deduced from \eqref{solutionp0pm},
where $c$ is a normalisation found through 
\eqref{normilzep}.
Finally,
\begin{equation*}
L^{(1)}(x)
=
\int_{0}^{x} e^{K(x-s)} K^{(1)} e^{Ks} \mathrm{d}s, 
\end{equation*}
where 
$K^{(1)}=C_{+}^{-1} \tilde{Q}_{++}
+\Psi^{(1)} |C_{-}^{-1}| Q_{-+}
+ \Psi |C_-^{-1}| \tilde{Q}_{-+}$.
In order to actually compute $L^{(1)}(x)$,
we refer the reader to Higham \cite[Theorem 10.13, Equation (10.17a)]{higham2008functions}.

\section{Numerical Illustration}\label{Illustr}

We evaluate and display the value of
$E_\infty(\varepsilon) = \parallel \boldsymbol{\Psi}(\varepsilon) - \overline\Psi -
\varepsilon \Psi^{(1)} \parallel_\infty$
for a few examples where only the phases in $\mathcal{S}_{0}$ are
perturbed, either by a positive quantity as in
Section~\ref{AffectedCase+}, or by a negative quantity, as in
Section~\ref{Casemigrto-}.  The narrative is as follows: assume that
the rates $c_i$ in $\mathcal{S}_{0}$ are very small and positive (or
very small and negative) and that they are set equal to 0.  What is
the effect on the matrix $\Psi$?

The controlling phase evolves as a birth-and-death process on the
state space $\{1 \ldots 3m\}$ where
\begin{itemize}
\item phases 1 to $m$ belong to $\mathcal{S}_{+}$ and all have the
  same positive rate $r_+$;
\item phases $m+1$ to $2m$ belong to $\mathcal{S}_{0}$; when
  perturbed, they all have the same perturbation coefficient $\tilde
  r_0$, either positive or negative;
\item phases $2m+1$ to $3m$ belong to $\mathcal{S}_{-}$ and all have
  the same negative rate $r_-$.
\end{itemize}
The parameters $r_+$ and $r_-$
are chosen in all cases such that the stationary drift of the non-perturbed fluid
model is equal to $-0.1$.

\subsection*{Migration of $\mathcal{S}_{0}$ to $\mathcal{S}_{+}$.}

\paragraph{Case 1.a} The infinitesimal generator is that of the M/M/1/N
queue with $N=3m$, that is,
\begin{equation}
   \label{e:Acase1}
A=\left[\begin{array}{cccccc}
-\lambda & \lambda & 0\\
\mu & -(\lambda+\mu) & \lambda\\
0 & \mu & -(\lambda+\mu) & \ddots\\
 &  & \ddots & \ddots\\
 &  &  &  & -(\lambda+\mu) & \lambda\\
 &  &  &  & \mu & -\mu
\end{array}\right].
\end{equation}
We assume that $\lambda > \mu$, so that the process
spends most of its time in $\mathcal S_-$ in this case.
In our experimentation, we have noticed that the quantities
\[
E_+(\varepsilon) = \max_{i \in \mathcal S_+} \sum_{j \in \mathcal S_-}
|\boldsymbol{\Psi}(\varepsilon) - \overline\Psi -
\varepsilon \Psi^{(1)} |_{ij}
\]
and 
\[
E_\oplus(\varepsilon) = \max_{i \in \mathcal S_0} \sum_{j \in \mathcal S_-}
|\boldsymbol{\Psi}(\varepsilon) - \overline\Psi -
\varepsilon \Psi^{(1)} |_{ij}
\]
are significantly different sometimes, and for that reason we give
separately their values in the figures of this section, the norm
$E_\infty$ is easily found as the maximum of $E_+$ and $E_\oplus$.

\begin{figure}
\centering
\includegraphics[scale=0.4]{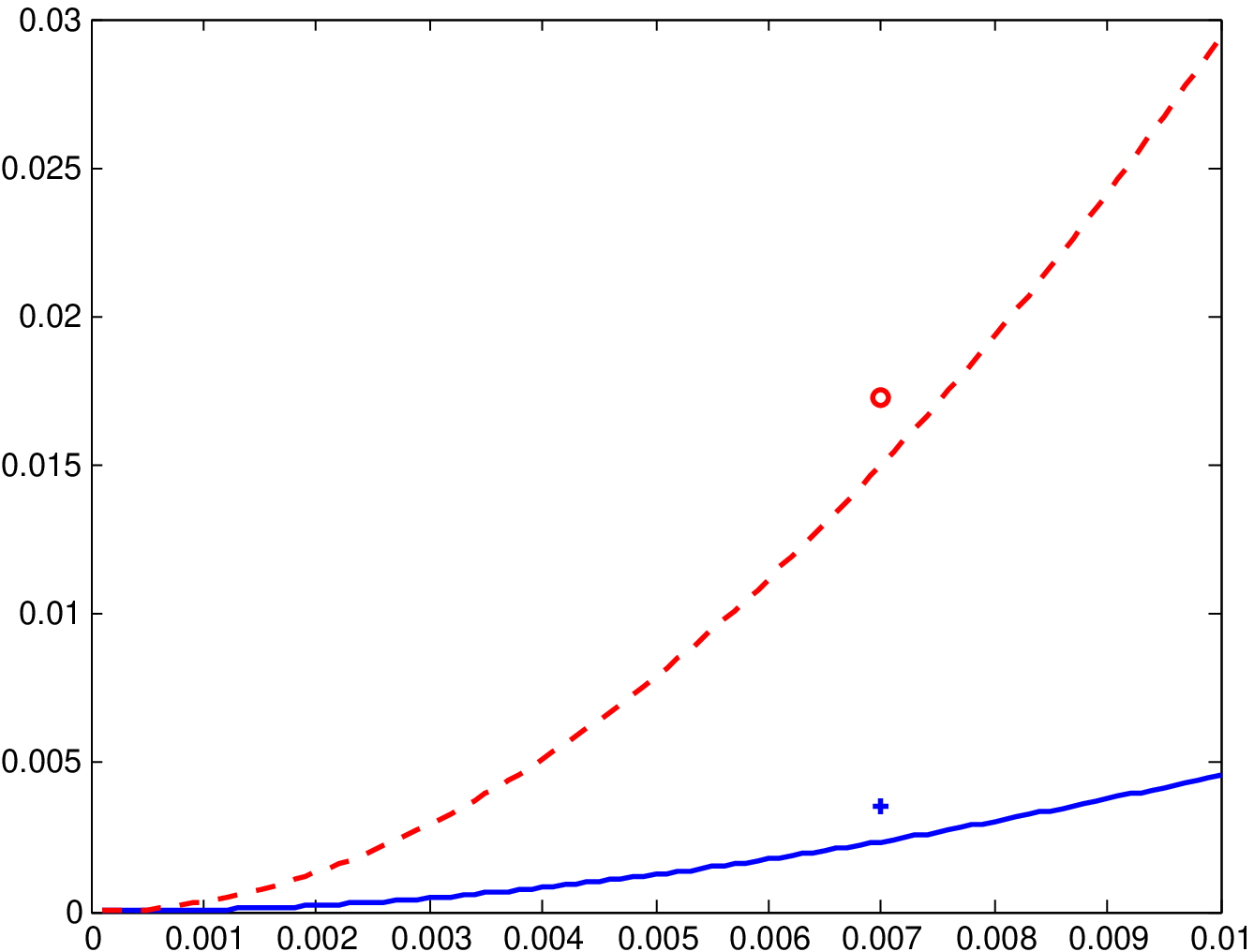}
\qquad \includegraphics[scale=0.4]{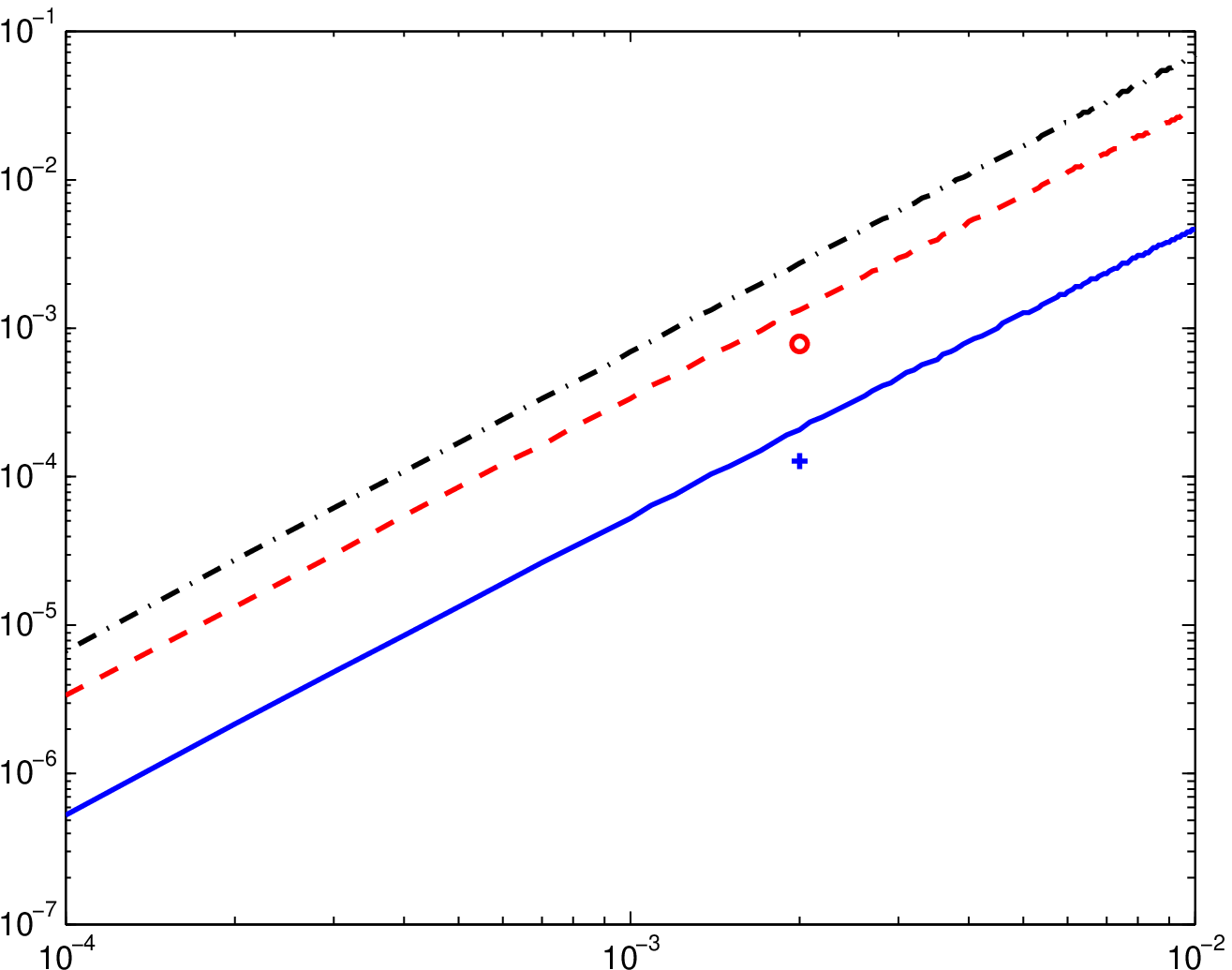}

\caption{$E_+(\varepsilon)$ and $E_\oplus(\varepsilon)$
  for Case 1.a,  $\varepsilon = 10^{-4}$ to $10^{-2}$; $m=5$,
$\lambda=2$, $\mu=1$, $r_{+}=0.4$, $r_{-}=-0.207$, the perturbation is
$\tilde r_{\oplus}=0.4$.   
   \label{f:case1A}
}
\end{figure}

The results on Figure~\ref{f:case1A} are obtained from
$\lambda =2 > \mu =1 $, $m=5$, $r_+ = 0.4$, $r_- = -0.207$,
$\tilde r_\oplus = r_+ >0$.  On the left, $E_+(\varepsilon)$ is
displayed as a continuous line, marked with a '+' sign,
$E_\oplus (\varepsilon)$ is displayed as a dashed line, marked with an
'o' sign.  On the right of Figure~\ref{f:case1A}, we display the same
functions on a logarithmic scale; in addition, we include for visual
reference a function proportional to $\varepsilon^2$, as the lined
marked with alternating dashes and dots.

The logarithm  plot shows in a striking manner that the difference
$\boldsymbol\Psi(\varepsilon) 
- \overline\Psi - \varepsilon \Psi^{(1)} $ is $O(\varepsilon^2)$.
This may also be seen in  Table~\ref{t:cases1to3}, where we give the
values of $E_+$ and $E_\oplus$ for $\varepsilon = 10^{-4}$ and
$\varepsilon = 10^{-2}$, for Cases 1.a, 2.a and 3.a.

\begin{table}
\centering  
\begin{tabular}{c|cc|cc}
  & \multicolumn{2}{c|}{$\varepsilon = 10^{-4}$} &
             \multicolumn{2}{c}{$\varepsilon = 10^{-2}$} \\
Case   & $E_+$ & $E_\oplus$ & $E_+$ & $E_\oplus$ \\
\hline
1.a &  5.37 $10^{-7}$  &  3.39 $10^{-6}$  &  4.60 $10^{-3}$  &  2.94 $10^{-3}$ \\
2.a & 1.92 $10^{-12}$  & 2.00 $10^{-12}$  &  2.08 $10^{-8}$  & 2.15 $10^{-8}$ \\
3.a & 3.77 $10^{-8}$   & 4.80 $10^{-8}$   &  3.66 $10^{-4}$  &  4.67 $10^{-4}$ \\
\end{tabular}
\caption{Values of $E_+(\varepsilon) $ and $E_{\oplus}(\varepsilon) $ in Cases 1.a to 3.a, for
  $\varepsilon$ equal to $10^{-4}$ and~$10^{-2}$.
   \label{t:cases1to3}
}
\end{table}

\paragraph{Case 2.a} The infinitesimal generator for the phase is given
by (\ref{e:Acase1}), the
same as in Case 1.a, but here we take $\lambda<\mu$, so that the process
spends most of its time in $\mathcal S_+$.  The parameters in this
case are $\lambda =1 < \mu =2$,
$m=5$, $r_+ = 0.4$, $r_- = -621$, $\tilde r_\oplus = r_+ >0$.   Notice that
we must use a very small value for $r_-$, to compensate for the time
spent in $\mathcal S_+$  and keep $-0.1$  as the
stationary drift.

The results for this case are displayed on the left of
Figure~\ref{f:case2n3} and it appears that $E_+$ and $E_\oplus$ are
very close to each other.  Furthermore, they are much smaller that in
Case 1.a.  This is very clear from Table~\ref{t:cases1to3}, where we
observe that $E_\infty$ is
several orders of magnitude smaller in Case 2.a than in Case~1.a.

\begin{figure}
\centering
\includegraphics[scale=0.4]{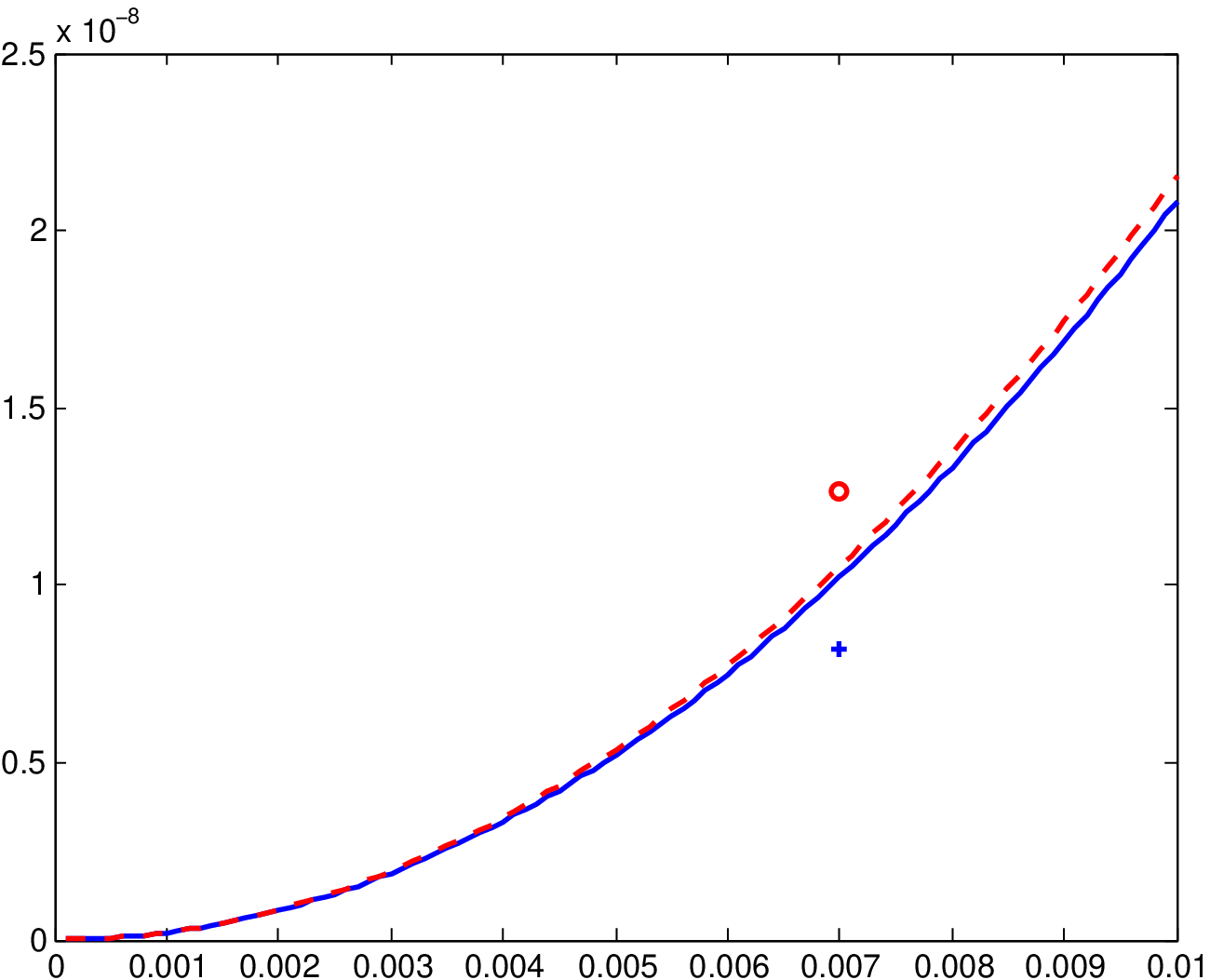}
\qquad
\includegraphics[scale=0.4]{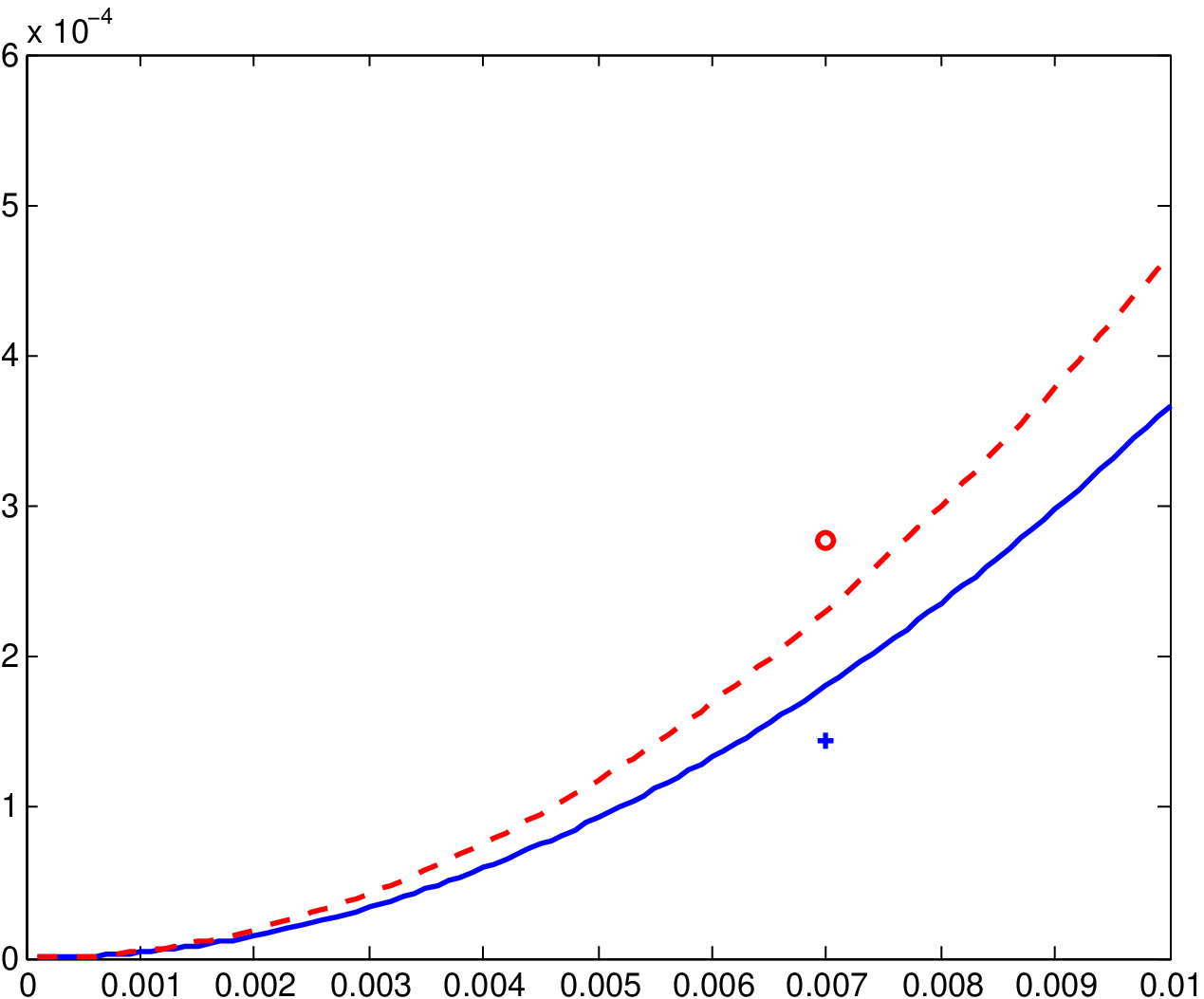}
\caption{Case 2.a is displayed on the left, with parameters $m=5$,
$\lambda=1$, $\mu=2$, $r_{+}=\tilde r_{\oplus}=0.4$, $r_{-}=-621$.
Case 3.a is displayed on the right, with parameters $m=5$, $\alpha=\beta=1$,
$r_{+}=\tilde r_{\oplus}=0.4$, $r_{-}=-2.63.$
   \label{f:case2n3}
}
\end{figure}

\paragraph{Case 3.a}  The infinitesimal generator in Case 3.a is that of a
system of $N$ individuals independently alternating between two
states.  It is given by
\begin{equation*}
A=\left[\begin{array}{cccccc}
-d_{0,1} & (N-1)\alpha & 0\\
\beta & -d_{1,2} & (N-2)\alpha\\
0 & 2\beta & -d_{2,3} & \ddots\\
 &  & \ddots & \ddots\\
\\
 &  &  &  & -d_{N-2,N-1} & \alpha\\
 &  &  &  & (N-1)\beta & -d_{N-1,N}
\end{array}\right],
\end{equation*}
where $d_{i,j}=i\beta+(N-j)\alpha$ and $N=3m$.  We take $\alpha =
\beta$ so that the distribution is concentrated in the middle of the
range $[1\ldots N]$, that is, in the region covered by $\mathcal S_0$.

The results are given on the right in Figure~\ref{f:case2n3} and in
the last row of Table~\ref{t:cases1to3}, the parameters are $m=5$,
$\alpha = \beta =1$, $r_+=0.4$, $r_-=-2.63$ and $\tilde r_\oplus=r_+$.

\subsection*{Migration of $\mathcal{S}_{0}$ to $\mathcal{S}_{-}$}

In the second set of examples, labeled Cases 1.b to 3.b, we take the same
parameters as in Cases 1.a to 3.a, except that the perturbation for the
phases in $\mathcal S_0$ are negative, and we take in each case
$\tilde r_\ominus = -r_+$.  Here, there is only one set of functions
$E(\cdot)$ to compute: the rows of $\Psi$ are all labeled by phases
in $\mathcal S_+$ and so $E_\infty=E_+$, while $E_\oplus$ is not defined.

Graphically, we have observed results very similar to those in
Figures~\ref{f:case1A} and~\ref{f:case2n3} and we do not give the
graphs here.  Instead, we give in Table~\ref{t:setB} the values of
$E_\infty(\varepsilon)$ for $\varepsilon = 10^{-4}$ and $10^{-2}$.
The obtained values are similar to those obtained for Cases 1.a to 3.a
and it appears clearly that $E_\infty(\varepsilon)$ is
$O(\varepsilon^2)$.

We might also compute the 1-norm of $\boldsymbol{\Psi}(\varepsilon) - \overline\Psi
-\varepsilon \Psi^{(1)}$  instead of its $\infty$-norm,  and compare
the two partial norms
\[
E_-(\varepsilon) = \max_{j \in \mathcal S_-} \sum_{j \in \mathcal S_+}
| \boldsymbol{\Psi} (\varepsilon) - \overline\Psi -
\varepsilon \Psi^{(1)} |_{ij}
\]
and 
\[
E_\ominus(\varepsilon) = \max_{j \in \mathcal S_0} \sum_{i \in \mathcal S_+}
| \boldsymbol{\Psi} (\varepsilon) - \overline\Psi -
\varepsilon \Psi^{(1)} |_{ij},
\]
with $\parallel \boldsymbol{\Psi}(\varepsilon) - \overline\Psi -\varepsilon
\Psi^{(1)} \parallel_1$ = $\max(E_-(\varepsilon), E_\ominus(\varepsilon))$.
We would expect to observe differences similar to those between $E_+$
and $E_\oplus$ in Cases 1.a to 3.a.


\begin{table}
\centering  
\begin{tabular}{c|cc}
Case    & $\varepsilon = 10^{-4}$ &
             $\varepsilon = 10^{-2}$ \\
\hline
1.b &  1.08 $10^{-7}$  &  1.05 $10^{-3}$  \\
2.b & 5.11 $10^{-8}$  & 4.91 $10^{-4}$ \\
3.b & 1.33 $10^{-6}$   & 1.15 $10^{-2}$   \\
\end{tabular}
\caption{Values of $E_\infty(\varepsilon) $  in Cases 1.b to 3.b, for
  $\varepsilon$ equal to $10^{-4}$ and $10^{-2}$.
   \label{t:setB}
}
\end{table}

\section*{Acknowledgement}
This work was supported in part by 
the Minist\`{e}re de la Communaut\'{e} fran\c caise de Belgique through the ARC grant
AUWB-08/13-ULB 5 and in part by the Flemish Community of Belgium
through the Methusalem program.

%
%
%
\end{document}